\documentclass{article}
\usepackage{amsthm}
\usepackage{systeme}
\usepackage{geometry}%
\usepackage[pagebackref]{hyperref}%
\usepackage{lmodern,mathtools, amssymb}%
\usepackage[utf8]{inputenc}%
\usepackage[T1]{fontenc}%
\usepackage{todonotes}%
\usepackage{color}%
\usepackage{dsfont}
\usepackage{pgf,tikz,pgfplots}
\pgfplotsset{compat=1.14}
\usepackage{mathrsfs}
\usetikzlibrary{arrows}
\usepackage{array}
\newcolumntype{P}[1]{>{\centering\arraybackslash}p{#1}}

\newtheorem{definition}{Definition}
\newtheorem{theorem}{Theorem}[section]%
\newtheorem{proposition}[theorem]{Proposition}%
\newtheorem{lemma}[theorem]{Lemma}%
\newtheorem{remark}[theorem]{Remark}%

\title{Strong solutions to a beta-Wishart particle system }
\author{
  Benjamin Jourdain\footnote{CERMICS, Ecole des Ponts, INRIA, Marne-la-Vall\'ee, France. Emails: benjamin.jourdain@enpc.fr, ezechiel.kahn@enpc.fr}
  \and
   Ez\'echiel Kahn$^*$
}
\date{\today}
\begin{document}

\maketitle

\begin{abstract}
    The purpose of this paper is to study the existence and uniqueness of solutions to a Stochastic Differential Equation (SDE) coming from the eigenvalues of Wishart processes.   The coordinates  are non-negative, evolve as Cox-Ingersoll-Ross (CIR) processes and repulse each other according to a Coulombian like interaction force.  We show the existence of strong and pathwise unique solutions to the system until the first multiple collision, and give a necessary and sufficient condition on the parameters of the SDE for this multiple collision not to occur in finite time. 
\end{abstract}

\section{Introduction}

Let $(M_t)_t$ be a stochastic process tacking its values in the space of  $m\times n$ matrices with real entries verifying the following stochastic differential equation
\begin{equation}
    dM_t = \kappa_1 dW_t-\kappa_2 M_tdt, \text{ }M_0=m_0\nonumber
\end{equation}
where $W$ is a $m\times n$ matrix filled with independant Brownian motions, $m_0$ is a $m\times n$ deterministic matrix, $\kappa_1\in\mathbb{R}$ and $\kappa_2\in\mathbb{R}_+$. The entries of the matrix $M$  are  independent Ornstein-Ulhenbeck processes just as the one considered in \cite{MR1132135}. Such a process is called a \emph{Wishart process}, and it was shown in  \cite{MR991060} and \cite{MR1132135} that the eigenvalues of $M^\dagger M$ satisfy the system of SDEs
\begin{equation*}
    d\lambda^i_t=2\sqrt{\lambda^i_t}dB^i_t+m\kappa_1^2dt-2\kappa_2\lambda^i_tdt+\sum_{j\neq i}\frac{\lambda^i_t+\lambda^j_t}{\lambda^i_t-\lambda^j_t}dt \text{ for all }i\in\{1,\dots n\}
\end{equation*}
where $B^1,\dots, B^n$ are independent Brownian motions. The reader will find in \cite{MR1871699} an analysis of the complex analog of Bru's model. In this paper, we aim to prove existence and uniqueness of solutions to such systems of SDEs for a broader range of parameters.

Let $\alpha\geq0,$ $\gamma\in\mathbb{R}$, $\beta>0,$ $n\geq2$, and $\mathbf{B}=(B^1,\dots,B^n)$ be a $n$-dimensional Brownian motion. Our SDE system of interest is the following :

\begin{eqnarray}
\label{eds}
d\lambda_t^{i}& =& 2\sqrt{\lambda_t^{i}}dB_t^{i} + \left( \alpha-2\gamma\lambda_t^{i}+ \beta\sum_{j\ne i}\frac{\lambda_t^{i}+\lambda_t^{j}}{\lambda_t^{i}-\lambda_t^{j}} \right)dt   \\
&&0\leq \lambda^1_t<\dots< \lambda^n_t, \text{  a.s. } dt-\text{almost everywhere}. \label{orderr} 
\end{eqnarray}

 The system (\ref{eds}) describes the positions of $n$ ordered particles evolving in $\mathbb{R}_+$. It can be rewritten
 
 \begin{eqnarray}
d\lambda_t^{i} & = & 2\sqrt{\lambda_t^{i}}dB_t^{i} + \left( \alpha-(n-1)\beta-2\gamma\lambda_t^{i}+ 2\beta\lambda_t^i\sum_{j\ne i}\frac{1}{\lambda_t^{i}-\lambda_t^{j}} \right)dt\text{ for all  } i\in \{1,\dots,n\}, \label{edsbis}\\
&&0\leq \lambda^1_t<\dots< \lambda^n_t, \text{ a.s. } dt\text{-a.e.} \nonumber.
\end{eqnarray}
We will look for continuous solutions to the SDE (\ref{eds}). Thus, by continuity, we have for all $t\geq0$
\[0\leq \lambda^1_t\leq\dots\leq \lambda^n_t \text{ a.s..}\]

\paragraph{}


For all $i$, let $x_t^{i}=\sqrt{\lambda_t^{i}}$ and $X=(x^1_t,\dots,x^n_t)_t$. We apply  formally (as $x\mapsto\sqrt{x}$ is not twice continuously differentiable in $0$) Ito's formula to obtain
\begin{eqnarray}
\label{sqrteds}
dx_t^{i} &=&dB_t^{i}+ \left(\frac{\alpha-1}{2}\frac{1}{x_t^{i}}-\gamma x_t^{i} +\frac{\beta}{2x_t^{i}}\sum_{j\ne i}\frac{(x_t^{i})^2+(x_t^{j})^2}{(x_t^{i})^2-(x_t^{j})^2}\right)dt \\
& = & dB_t^{i}+ \left(\frac{\alpha-(n-1)\beta-1}{2}\frac{1}{x_t^{i}}-\gamma x_t^{i} +\beta x_t^{i}\sum_{j\ne i}\frac{1}{(x_t^{i})^2-(x_t^{j})^2}\right)dt \text{ for all  $i\in \{1,\dots,n\}$} \label{sqrtedsdual}\\
&&0\leq x^1_t<\dots< x^n_t, \text{ a.s., } dt-a.e..\nonumber
\end{eqnarray}
When $X$ is a solution to (\ref{edsbis}) then $\Lambda =((x^1_t)^2,\dots, (x^n_t)^2)_t$ is a solution to (\ref{eds}), but it is not true the other way round.

The system (\ref{sqrteds}) can be rewritten  
 \begin{equation}
 \label{edspot}
     dx_t^{i}=dB_t^{i}- \partial_iV(x_t^1, ..., x_t^n)dt\text{ for all  $i\in \{1,\dots,n\}$}   
 \end{equation}
 with
 \begin{equation}
 \label{potentiam}
  V(x_1,...,x_n)=- \sum_{i=1}^n\left\{\frac{\alpha-(n-1)\beta-1}{2}\ln| x^{i}|-\frac{1}{2}\gamma (x^{i})^{2}+\frac{\beta}{4}\sum_{j\ne i}\left(\ln| x^{i}-x^{j}| + \ln| x^{i}+x^{j}|\right)\right\}.
 \end{equation}
  
  Indeed, for all $i\in\{1,\dots,n\}$ and for all $j\neq i$, $x^i_t\neq x^j_t$ :
  \begin{eqnarray}
      \frac{1}{x_t^{i}}\sum_{j\ne i}\frac{(x_t^{i})^2+(x_t^{j})^2}{(x_t^{i})^2-(x_t^{j})^2} 
      & = & -\frac{n-1}{x_t^{i}}+\sum_{j\ne i}\left(\frac{1}{x_t^{i}-x_t^{j}}+\frac{1}{x_t^{i}+x_t^{j}}\right). \label{potentialcomputation}
  \end{eqnarray}
  
\paragraph{}
The difficulty in proving the existence of solutions to such   SDEs comes from the fact that there are singularities both when a particle touches zero, as made more explicit in (\ref{sqrteds}) by the square root change of variables, and when two particles touch each other. Both events are called "collision" from now on, and we will speak about "collision between particles" when two particles touch each other.  Consequently, it is not enough to show that there is no collision between the particles to prove  the existence, as it is the case in \cite[Theorem 4.3.2 p251]{MR2760897} for the Dyson Brownian motions which satisfy up to a change of time
\begin{equation*}
    d\lambda^{\mathcal{D},i}_t = \sqrt{2}dB^i_t+\tilde\beta\sum_{j\neq i}\frac{dt}{\lambda^{\mathcal{D},i}_t-\lambda^{\mathcal{D},j}_t}, \text{ for all }i\in\{1,\dots,n\}.
\end{equation*}

If we define $D=\{0<\lambda^1<\lambda^2<\dots<\lambda^n\}$, a collision occurs when the process $\Lambda$ hits the boundary $\partial D$ made of the union of $\{\lambda^i=\lambda^{i+1}\}$ for $i\in\{1,\dots,n-1\}$ and $\{\lambda^1=0\}$. Then, a multiple collision occurs when two of these sets are reached at the same time.
\paragraph{}
Our results about the SDE (\ref{eds}) are the following. In Proposition \ref{multiplecollision}, we give a necessary and sufficient condition for a multiple collision in zero of the coordinates of a solution to the SDE (\ref{eds}) to occur in finite time. Our main result Theorem \ref{existence} give existence and uniqueness of solutions to the SDE (\ref{eds}) on  broader ranges of parameters than it was done before to the best of our knowledge. In Proposition \ref{everyparticlecollide} we give a condition for every coordinate of a solution to the SDE (\ref{eds}) to collide with one of its neighbours in finite time, and we give in Proposition \ref{stationary}   the unique stationary probability measure of the SDE (\ref{eds}).

\paragraph{}

The paper is organized as follows. The rest of Section 1 is devoted to the bibliographical background of this work. In Section 2, we state our main results, in particular Proposition \ref{multiplecollision} and Theorem \ref{existence}. We prove in Section 3 some useful properties of the solutions to the SDE (\ref{eds}), before proving  in Section 4  Proposition \ref{multiplecollision} and Theorem \ref{existence}. We prove the rest of the results in Section 5. Section 6 is an Appendix stating some well-known results that we use in our proofs.

\paragraph{}
Systems of interacting particles following equations of the type
\begin{equation}
\label{example}
    dZ^i_t = b_n(Z^i_t)dt + \sigma_n (Z^i_t)dB^i_t-\psi^i(Z^1_t,\dots,Z^n_t)dt
\end{equation}
where the $\psi^i$ are singular repulsive interaction functions have been studied by many authors.

Rogers and Shi \cite {MR1217451} were interested into the asymptotic behaviour  when $n$ goes to infinity of the empirical measure of the particles solutions to (\ref{example})  when $\psi^i$ takes the form of the $i$-th derivative of a potential.

When this potential takes the form
\begin{equation}
   \Psi : z\in\{z\in\mathbb{R}_+^n:z_1\leq\cdots\leq z_n\}
    \mapsto\ln\left(\prod_{i=1}^nz_i^{\beta(m-n+1)-1}
    \mathrm{e}^{-\frac{\gamma}{2}z_i^2}\prod_{1\leq i<j\leq n}(z_j^2-z_i^2)^\beta\right),\nonumber
  \end{equation} 
  for $m\in\mathbb{N}$ and $\sigma_n=1$,
  the solution to the equation takes the name of  $\beta$-Laguerre process. The parameters $n$ and $m$ are then respectively the number of lines and columns of the underneath random matrix model, and $\beta=1,2$ or $4$ depending on the dimension of the underlying algebra ($\mathbb{R}$, $\mathbb{C}$ or $\mathbb{H}$). The formula still makes sense for all $\beta>0$. The Boltzmann-Gibbs measures related to these potentials were documented by Forrester in \cite{MR2641363}. 
  \paragraph{Link with the multivalued stochastic differential equations theory.}
The systems of type (\ref{example}) were  deeply studied by Cépa and Lépingle   for instance in \cite{MR1440140} and \cite{MR1875671}  where they apply Cépa's multivalued stochastic differential equations theory developed in  \cite{MR1459451}. This theory treats existence and uniqueness of solutions to multivalued SDEs associated with a convex function defined on a domain of $\mathbb{R}^n$. In  \cite{MR2792586}, Lépingle applied this theory to a constrained Brownian motion between reflecting or repellent walls of  Weyl chambers. The boundary behavior of the convex function dictates the behavior of the process  on these same boundaries (hitting or not the boundary in finite time, reflection on the boundary...). Our SDE of interest rewritten in the form (\ref{edspot}) thanks to the square root change of variables can be seen this way.

\paragraph{Link with radial Dunkl processes.}
Let us define a reduced root system $R$ by a finite set in $\mathbb{R}^n\backslash \{0\}$ and  $V=Span(R)$ such that
\begin{itemize}
    \item for all $\alpha\in R$, $R\cap\mathbb{R}\alpha=\{\alpha,-\alpha\}$,
    \item  for all $\alpha\in R$, $\sigma_\alpha(R)=R$ 
\end{itemize}
where $\sigma_\alpha$ is the reflection with respect to the hyperplane orthogonal to $\alpha$. A simple system $\Delta$ is a basis of $V$, and induces a total ordering in $R$ the following way : a root $\alpha\in R$ is  positive if it is a positive linear combination of elements of $\Delta$. We can thus define $R_+$ as the set of positive roots of $R$.
When $\sigma_n = 1$ and $\Psi$ takes the form 
\begin{equation}
    \Psi : z\mapsto -\sum_{\alpha\in R_+}k(\alpha)\ln(\langle\alpha,z\rangle), \text{ }z\in C,\nonumber
\end{equation}
where $C$ is the positive Weyl chamber defined by 
\[C=\{x\in V, \langle \alpha,x\rangle>0\text{ }\forall\alpha\in R_+\},\]
 and $\psi^i=\partial_i\Psi$, Demni proved in \cite[Theorem 1]{MR2566989} existence and uniqueness of a solution to (\ref{example}) on the domain $\bar C$ when $k(\alpha)>0$ for all $\alpha\in R_+$. To do so, he applied Cépa's multivalued stochastic differential equations theory. This system corresponds to (\ref{eds}) for some choice of $R_+$ and $k$. 
 Indeed, when the root system is of so-called $B_n$-type,  it is defined by 
 \begin{align}
     R &= \{\pm e_i, \pm e_i \pm e_j, 1\leq i<j\leq n\}, \nonumber\\
     \Delta &= \{e_{i+1}-e_{i}, 1\leq i \leq n-1 , e_1\},\nonumber\\
     R_+ & = \{e_i, 1\leq i\leq n, e_j\pm e_i, 1\leq i<j\leq n\},\nonumber\\
     C &= \{x\in\mathbb{R}^n, 0<x^1<\dots <x^n\}, \nonumber
 \end{align}
 which,  with the right choice of $k$   gives equation (\ref{sqrteds}) with $\gamma=0$ (see (\ref{potentiam}) and (\ref{potentialcomputation}) for the computation). The condition $k(\alpha)>0$ for all $\alpha\in R_+$ implies $\alpha-(n-1)\beta>1$.
 We seek here to obtain the existence of a solution to (\ref{eds}) while relaxing the last inequality.
 
\paragraph{Link with other works.}

The reader will find in Graczyk and  Malecki  \cite{MR3076363} and \cite{MR3296535} a treatment of equations  (\ref{example}) when $\psi^i$ takes the form 
\begin{equation}
    \psi^i : (z_1,\dots,z_n)\mapsto \sum_{j\neq i}\frac{H_{i,j}(z_i,z_j)}{z_i-z_j}.\nonumber 
\end{equation}
The SDE (\ref{eds}) is a subcase of these equations when $H_{i,j}(z_i,z_j) = z_i+z_j$. In \cite[Theorem 2.2]{MR3296535}, the authors studied these SDEs    and demonstrated the  existence of a strong solution on the time interval $[0,+\infty)$ when $\beta\geq1$. In this regime, they proved that there is no collision between the particles. The authors also demonstrated the  pathwise uniqueness of the solutions to this system for every $\beta$, as recalled in Lemma \ref{pathwise_uniqueness}.

\paragraph{}

In all these references, $\beta$ is identified as a fundamental parameter, its position relative to $1$ governing the possibility of collisions between the particles.

\section{Main results}
We start this section by giving a standard definition that we will use in the paper.
\begin{definition}
Let  $\Lambda_0$ be independent from the Brownian motion $\mathbf{B}$ and let us define for all $t\geq0$ $\mathcal{F}_t=\left(\Lambda_0,(\mathbf{B}_{s})_{s\leq t}\right)$. We will say that a solution to a SDE with initial condition $\Lambda_0$ is \textbf{global} if it is defined on the whole time interval $\mathbb{R}_+$, and \textbf{local} if it is defined up to a stopping time $\mathcal{T}$ for the filtration $\left(\mathcal{F}_t\right)_{t\geq0}$.
\end{definition}

\paragraph{}
The  rewritting (\ref{edsbis}) of the SDE hints that $\alpha-(n-1)\beta$ is a fundamental parameter impacting the existence of solutions. Consequently, we will study where this coefficient has to lie for the SDE to have a  solution. For instance, if we assume  $\alpha-(n-1)\beta<0$, we have:

\begin{eqnarray}
\label{eds1}
d\lambda_t^{1}& =&2\sqrt{\lambda_t^{1}}dB_t^{1} + \left( \alpha-(n-1)\beta-2\gamma\lambda_t^{1}+ 2\beta\lambda_t^1\sum_{j\ne 1}\frac{1}{\lambda_t^{1}-\lambda_t^{j}} \right)dt \\
&\leq&2\sqrt{\lambda_t^{1}}dB_t^{1}-2\gamma\lambda_t^{1}dt.\nonumber
\end{eqnarray}
Then, according to the pathwise comparison theorem of Ikeda and Watanabe (that we recall in Theorem \ref{Ikeda} below) :  $$\lambda^1_t\leq r_t\text{ a.s. for all }t\geq0 $$ where  $$r_t=\lambda^1_0+ 2\int_0^t\sqrt{r_s}dB^1_s-2\int_0^t\gamma r_sds \text{ for all }t\geq0$$
which is a CIR process (see for instance  \cite[Theorem 6.2.2]{MR2362458}). By standard results on CIR processes recalled in Lemma \ref{lamberton} we can conclude that the stopping time $T=\inf\{t\geq0 : r_t=0\}$ verifies $\mathbb{P}(T<\infty)=1$ for $\gamma\geq0$, and $0<\mathbb{P}(T<\infty)<1$ for $\gamma<0$. On $\{T<\infty\}$, after $T$,  $r$ stays at zero indefinitely. As the drift in (\ref{eds1}) is strictly negative when $\lambda^1_t=0$ and therefore when $r_t=0$, it will stay strictly negative on a time interval of positive measure. Consequently, the SDE has no global in time solution.

We thus proved the following result:
  \begin{remark}
\label{0}
 A necessary condition for the existence of a global in time solution to (\ref{eds}) is $\alpha-(n-1)\beta\geq0$.
\end{remark}

We will thus assume this condition in the remaining of the paper.

It is proved in \cite[Corollary 8]{MR3076363} with condition $0\leq\lambda^1_0<\dots<\lambda^n_0$
that for  $\beta\geq1$, the SDE (\ref{eds}) has a unique global strong solution   and that there  actually is no collision.

Demni proved in  \cite[subsection 5.1]{MR2566989}, applying \cite[Theorem 1]{MR2566989} that under the conditions $\alpha -(n-1)\beta>1$ , $\beta>0$, $\gamma=0$ and $0< \lambda^1_0<\dots<\lambda^n_0$, the SDE (\ref{sqrteds}) admits a unique strong solution, and that for $0<\beta<1$ and $\alpha -(n-1)\beta>1$, there is collision between any neighbour particles $\lambda^i$ and $\lambda^{i+1}$ for $i\in\{1,\dots,n-1\}$. In Theorem \ref{existence} and Proposition \ref{everyparticlecollide}, we tackle respectively the existence problem and the collision problem on  wider  intervals for the parameters $\alpha-(n-1)\beta$ and $\gamma$.


   We first give a condition for $k$ particles to collide in zero. 

\begin{proposition}[Multiple collision in zero]
\label{multiplecollision}
Let $k\in\{1,\dots,n\}$. Let the initial condition $\Lambda_0=(\lambda^1_0,\dots, \lambda^n_0)$ be  independent from the Brownian motion $\mathbf{B}$ and such that $0\leq\lambda^1_0\leq\dots\leq\lambda^n_0$ a.s..
Then,

(i) if $\gamma\geq0$, (\ref{eds}) has a global in time solution $\Lambda = (\lambda^1_t,\dots,\lambda^n_t)_t$ and  $k(\alpha-(n-k)\beta)<2,$ \\
then
\[\mathbb{P}(\exists t\geq0 : \lambda_t^{1}+\lambda_t^{2}+\dots+\lambda_t^{k}=0)=1 ,\]

(ii) if $\gamma\in\mathbb{R}$,  $\lambda^k_0>0$ a.s., (\ref{eds}) has a local solution  $\Lambda = (\lambda^1_t,\dots,\lambda^n_t)_t$  defined up to a stopping time $\mathcal{T}$ and  $k(\alpha-(n-k)\beta)\geq2$,\\
then
\[ \mathbb{P}(\mathcal{T}=+\infty,\exists t\geq0 : \lambda_t^{1}+\lambda_t^{2}+\dots+\lambda_t^{k}=0)+\mathbb{P}(\mathcal{T}<+\infty,\underset{t\in[0,\mathcal{T})}{\inf} \lambda_t^{1}+\lambda_t^{2}+\dots+\lambda_t^{k}=0)=0.\]
\end{proposition}
\begin{remark}
\label{remarkmultiplecollision}
We can in fact prove that under the assumptions made in $(i)$ of Proposition \ref{multiplecollision} but with $\gamma<0$,  \[\mathbb{P}(\exists t\geq0 : \lambda_t^{1}+\lambda_t^{2}+\dots+\lambda_t^{k}=0)\in(0,1)\] using Lemma \ref{lamberton} point $3$.
\end{remark}
This proposition is proved in the beginning of Section 4.

\paragraph{}
    For $\beta<1$, we can define a solution to (\ref{sqrteds})   using Cépa's multivoque equations theory \cite{MR1459451}, but we need for this the convexity of the potential $V$, and thus the condition $\alpha-(n-1)\beta>1$,  which is what was made in \cite{MR2566989}.

The following result also applies in the case $\alpha-(n-1)\beta\leq1$ :
\begin{theorem}
\label{existence} Let us assume $\beta<1,\alpha-(n-1)\beta>0$.  Let the initial  condition  $\Lambda_0=(\lambda^1_0,\dots, \lambda^n_0)$ be independent from the Brownian motion $\mathbf{B}$ and such that $0\leq\lambda^1_0\leq\dots\leq\lambda^n_0$ a.s. and $\lambda^2_0>0 $ a.s..
 
 Then, the SDE (\ref{eds}) has a unique strong solution defined on the time interval $$[0,\underset{\epsilon\rightarrow0}{\lim}\zeta_\epsilon)$$ where, for  $\epsilon>0$, 
\begin{equation}
      \zeta_\epsilon=\inf\{t\geq0 :  \lambda^1_t\leq\epsilon \text{ and }\lambda^2_t-\lambda^1_t\leq\epsilon\}.\nonumber
  \end{equation}

  Moreover, 
  \begin{align}
  \text{(i) } &\text{ for }\gamma\in\mathbb{R},\text{ if }\alpha-(n-1)\beta\geq 1-\beta \text{ then }\underset{\epsilon\rightarrow0}\lim \zeta_\epsilon=\infty \text{ a.s.} \nonumber\\
  \text{(ii) }&\text{ for }\gamma\geq0,\text{ if }\alpha-(n-1)\beta< 1-\beta \text{ then }\underset{\epsilon\rightarrow0}\lim \zeta_\epsilon<\infty \text{ a.s.} \nonumber\\
  \text{(iii)  }&\mathbb{P}\Big\{\exists t\in (0,\underset{\epsilon\rightarrow0}{\lim}\zeta_\epsilon): \lambda^i_t = \lambda^{i+1}_t \text{ and } \lambda^j_t = \lambda^{j+1}_t \text{ for some }1\leq i<j\leq n-1\Big\}=0 \nonumber
  \end{align}
\end{theorem}
\begin{remark}
Applying Remark \ref{remarkmultiplecollision}, we can in fact prove that under the assumptions made in $(ii)$ of Theorem \ref{existence} but with $\gamma<0$,  \[\mathbb{P}(\underset{\epsilon\rightarrow0}{\lim}\zeta_\epsilon<\infty)\in(0,1)\].
\end{remark}
Theorem \ref{existence} is proved in Section 4. The disjunction $(i)-(ii)$ comes from the application of Proposition \ref{multiplecollision} with $k=2$.
\begin{remark}
This last result states existence and uniqueness of a strong solution to (\ref{eds}) defined on $\mathbb{R}_+$ for $1-\beta\leq\alpha-(n-1)\beta$ and on $[0,\underset{\epsilon\rightarrow0}\lim \zeta_\epsilon)$ if $0<\alpha-(n-1)\beta<1-\beta$. The next step would be to find how to prove the existence of the limit  $\Lambda_{\underset{\epsilon\rightarrow0}\lim \zeta_\epsilon}$  and how start back from it to define a solution on the whole interval $\mathbb{R}_+$ in this last case.
\end{remark}

Demni proved in \cite{MR2566989} that for $\beta<1$ and $\alpha-(n-1)\beta>1$, a collision between the particles $\lambda^{i}$ and $\lambda^{i+1}$  occurs in finite time almost surely for all $i\in\{1,\dots,n-1\}$. We strengthen here this result by showing that for $\beta<1$ and $\alpha-(n-1)\beta>0$,   every particle touches its neighbour particles in finite time almost surely.

\begin{proposition}\label{everyparticlecollide}
Let $\beta<1,\gamma\geq0,\alpha>0$. Let the initial  condition $\Lambda_0=(\lambda^1_0,\dots,\lambda^n_0)$ be independent from the Brownian motion $\mathbf{B}$ such that $0\leq\lambda^1_0\leq\dots\leq\lambda_0^n$. Let us assume  that there is a global in time solution to (\ref{eds}) (which is the case when $\alpha-(n-1)\beta\geq1-\beta$ and $\lambda^2_0>0$ a.s. according to Theorem \ref{existence}). Then for all $i\in\{2,\dots,n\}$, the stopping time \[T^{(i)}= \inf\{t>0:\lambda^i_t=\lambda^{i-1}_t\} \] is such that
\[\mathbb{P}\left\{T^{(i)}<\infty\right\}= 1.\]
\end{proposition}
To check this result, we prove that 
\begin{equation}
\label{intresult}
    \int_0^{+\infty}\lambda^1_sds=+\infty \text{ a.s.}
\end{equation}
using the next result.
\begin{proposition}\label{stationary}
Let us assume $\gamma>0$ and $\alpha-(n-1)\beta>0$. The unique stationary probability measure of the SDE (\ref{eds}) is $\rho_{inv}$ with density with respect to the Lebesgue measure 
\[d\rho_{inv}(\lambda^1,\dots,\lambda^n) = \frac{1}{2^n}\times\frac{1}{\sqrt{\lambda^1\times\dots\times\lambda^n}}\times\frac{1}{\mathcal{Z}}\times\prod_{i=1}^n\left((\lambda^i)^{\frac{\alpha-1-(n-1)\beta}{2}}e^{-\frac{\gamma}{2}\lambda^i}\prod_{j\neq i}|\lambda^j-\lambda^i|^{\beta/2}\right)\mathds{1}_{0\leq\lambda^1\leq\dots\leq\lambda^n}d\lambda^1\dots d\lambda^n\]
where
\[\mathcal{Z} = \int_{0\leq x^1\leq\dots\leq x^n}\exp\left\{-2V\left(x^1, \dots, x^n\right)\right\}dx^1\dots dx^n\] and $V$ is defined in (\ref{potentiam}) :

if a solution to (\ref{eds}) is such that the distribution of $\Lambda_t$ does not depend on $t$, then this distribution is $\rho_{inv}$. When $\Lambda_0$ is distributed according to $\rho_{inv}$ and is independent from the Brownian motion $\mathbf{B}$, the unique solution to (\ref{eds}) is such that for all $t\in\mathbb{R}_+$, $\Lambda_t$ is distributed according to $\rho_{inv}$, and it is a strong solution.
\end{proposition}

\begin{remark}
To prove (\ref{intresult}), one could consider applying the Ergodic Theorem, but we were not able to prove that the process defined by SDE (\ref{eds}) is a Markov process. The difficulty came from the choice of the state space: when $\lambda^2_0=0$, we do not know how to prove the existence of a solution to (\ref{eds}). If the state space is $\{0\leq\lambda^1\leq\dots\leq\lambda^n, \lambda^2>0\}$ we do not know how to prove that  for all $t>0$,  $\mathbb{P}(\lambda^2_t>0)=0$ when $\alpha-(n-1)\beta<1-\beta$.
\end{remark}

Table \ref{tab1} shows, for $\gamma\geq0$, the conditions on the coefficients of the SDE (\ref{eds}) for the existence of strong solutions.

\begin{table}
  \centering
  \begin{tabular}{|P{2cm}|P{5cm}|P{10cm}|}
    \hline
    $\alpha-(n-1)\beta$ & $\beta\geq1$          & $\beta<1$ \\ \hline
    $<0$                 & Defined on $[0,\underset{\epsilon\rightarrow0}{\lim}\inf\{t\geq0 : \lambda^1_t \leq\epsilon\})$, no collision, Proposition \ref{neg} & Defined on $[0,\underset{\epsilon\rightarrow0}{\lim}\inf\{t\geq0 : \lambda^1_t \leq\epsilon\})$  which is finite a.s., Proposition \ref{neg}    \\ \hline
    $\in(0,1-\beta)$                  & empty interval & Defined on $[0,\underset{\epsilon\rightarrow0}{\lim}\inf\{t\geq0 :  \lambda^1_t\leq\epsilon \text{ and }\lambda^2_t-\lambda^1_t<\epsilon\})$ , which is finite a.s., Theorem \ref{existence}    \\ \hline
    $\geq1-\beta$                  & Defined on $\mathbb{R}_+$, no collision between particles & Defined on $\mathbb{R}_+$, collision in finite time but no multiple collision in zero, Theorem \ref{existence}    \\ \hline
    $>1$                  & Defined on $\mathbb{R}_+$, no collision between particles & Defined on $\mathbb{R}_+$, collision between particles in finite time,  \cite{MR2566989}, Proposition \ref{everyparticlecollide}    \\ \hline
    $\geq2$                  & Defined on $\mathbb{R}_+$, no collision & Defined on $\mathbb{R}_+$, collision between particles in finite time, $\inf\{t>0:\lambda^1=0\}=+\infty$ a.s., Proposition \ref{multiplecollision}    \\ \hline
  \end{tabular}
  \newline\newline
  \caption{Conditions on the coefficients of SDE (\ref{eds}) for the existence of strong solutions when $\gamma\geq0$}\label{tab1}
\end{table}
\section{Properties of the solutions}

One can first remark, and it will be useful in several proofs, that the sum of the $n$ coordinates of a solution $(\lambda^1_t,\dots,\lambda^n_t)_{t\ge0}$ to SDE (\ref{eds}) follows a CIR process. Indeed,
\begin{align}
   d\left(\sum_{i=1}^n\lambda^i_t\right) & = 2\sum_{i=1}^n\sqrt{\lambda^i_t}dB^i_t-2\gamma\sum_{i=1}^n\lambda^i_tdt+n\alpha dt\nonumber\\
   & = 2\sqrt{\sum_{i=1}^n\lambda^i_t}dW_t-2\gamma\sum_{i=1}^n\lambda^i_tdt+n\alpha dt\label{sumcir}
\end{align}
where $W$   defined by $W_0=0$ and $dW_t=\sum_{i=1}^n\frac{\sqrt{\lambda^i_t}}{\sqrt{\sum_{j=1}^n\lambda^j_t}}dB^i_t$ is a Brownian motion according to Lévy's characterization.
\paragraph{}

The first part of the next Lemma is  proved in \cite[Theorem 5.3]{MR3296535} but we reproduce the proof for the sake of completeness.
\begin{lemma}
  \label{pathwise_uniqueness}
  Let $\gamma\in\mathbb{R}$. The solutions to (\ref{eds}) are pathwise unique.
  
  Moreover, if  $Z=(z^1_t,\dots,z^n_t)_t$ and $\tilde Z=(\tilde z^1_t,\dots,\tilde z^n_t)_t$ are two global in time solutions to (\ref{eds}) with the same driving Brownian motion and verifying 
 \[\mathbb{E}\left[\sum_{i=1}^nz^i_0+\tilde z^i_0\right]<+\infty,\] then for all $t\geq0$,
 
 \begin{equation}
    \sum_{i=1}^n\mathbb{E}|z^i_{t}-\tilde z^i_{t}| \leq\left(\sum_{i=1}^n\mathbb{E}|z^i_0-\tilde z^i_0|\right)\exp(-2\gamma t) .\label{contraction} 
\end{equation}
  \end{lemma}
  \begin{proof}

   Let $Z=(z^1,\dots,z^n)$ and $\tilde Z=(\tilde z^1,\dots ,\tilde z^n)$ be two solutions to (\ref{eds}) with the same random initial condition $Z_0=\tilde Z_0$ independent from the same driving Brownian motion $\mathbf{B}$, defined respectively on $[0,T]$ and $[0,\tilde T]$, where $T$ and $\tilde T$ are stopping times for a filtration $(\mathcal{F}_t)_{t\geq0}$ with respect to which $\mathbf{B}$ is a Brownian motion and $Z_0$ is $\mathcal{F}_0$-measurable.
   
    Let $M>0$ and \[\tau_M  =  \inf\{t\in[0,T\wedge\tilde T] :\tilde z^n_t+ z^{n}_t\geq M \}\] with the convention $\inf\emptyset= T\wedge\tilde T$.
  As $Z$ and $\tilde Z$ are continuous and assumed well defined on respectively $[0,T]$ and $[0,\tilde T]$,  $\tau_M\uparrow T\wedge\tilde T$ when $M\uparrow\infty$.

  Because of the square root diffusion coefficient, the local time of $z^i-\tilde z^i$ at $0$ is zero (\cite[Lemma 3.3 p389]{MR1725357}). Applying the Tanaka formula to the process $z^i-\tilde z^i$ stopped at $\tau_M$, and summing over $i\in\{1,\dots,n\}$, we get for $t\geq0$
  \begin{eqnarray}
    \sum_{i=1}^n|z^i_{t\wedge\tau_M}-\tilde z^i_{t\wedge\tau_M}| & = & \sum_{i=1}^n|z^i_0-\tilde z^i_0|+\sum_{i=1}^n\int_0^{t\wedge\tau_M}\text{sgn}(z^i_s-\tilde z^i_s)d(z^i_s-\tilde z^i_s)\nonumber\\
    & = &2\sum_{i=1}^n\int_0^{t\wedge\tau_M}|\sqrt{z^i_s}-\sqrt{\tilde z^i_s}|dB^i_s\nonumber\\
    & &+ \beta\int_0^{t\wedge\tau_M}\sum_{i=1}^n\text{sgn}(z^i_s-\tilde z^i_s)\sum_{j\neq i}\left(\frac{z^i_s+z^j_s}{z^i_s-z^j_s}-\frac{\tilde z^i_s+\tilde z^j_s}{\tilde z^i_s-\tilde z^j_s}\right)ds -2\gamma\int_0^{t\wedge\tau_M}\sum_{i=1}^n|z^i_s-\tilde z^i_s|ds\nonumber\\\label{concentration}
  \end{eqnarray}
  where $\text{sgn}(x)=1$ if $x>0$ and $\text{sgn}(x)=-1$ if $x\leq0$.
  As the process is stopped at $\tau_M$, the stochastic integrals have zero expectation and
  \begin{eqnarray}
    \sum_{i=1}^n\mathbb{E}|z^i_{t\wedge\tau_M}-\tilde z^i_{t\wedge\tau_M}| & = & \beta\mathbb{E}\int_0^{t\wedge\tau_M}\sum_{i<j}\left[\frac{z^i_s+z^j_s}{z^i_s-z^j_s}-\frac{\tilde z^i_s+\tilde z^j_s}{\tilde z^i_s-\tilde z^j_s}\right](\text{sgn}(z^i_s-\tilde z^i_s)-\text{sgn}(z^j_s-\tilde z^j_s))ds \nonumber\\ &&-2\gamma\mathbb{E}\int_0^{t\wedge\tau_M}\sum_{i=1}^n|z^i_s-\tilde z^i_s|ds.\nonumber
  \end{eqnarray}
  We have for all $i<j$ :
  \begin{eqnarray}
    \left[\frac{z^i_s+z^j_s}{z^i_s-z^j_s}-\frac{\tilde z^i_s+\tilde z^j_s}{\tilde z^i_s-\tilde z^j_s}\right](\text{sgn}(z^i_s-\tilde z^i_s)-\text{sgn}(z^j_s-\tilde z^j_s)) & = & 2\frac{z^j_s\tilde z^i_s-z^i_s\tilde z^j_s}{(z^i_s-z^j_s)(\tilde z^i_s-\tilde z^j_s)}(\text{sgn}(z^i_s-\tilde z^i_s)-\text{sgn}(z^j_s-\tilde z^j_s))\nonumber\\
    & = & 2\frac{z^j_s(\tilde z^i_s-z^i_s)+z^i_s(z^j_s-\tilde z^j_s)}{(z^i_s-z^j_s)(\tilde z^i_s-\tilde z^j_s)}(\text{sgn}(z^i_s-\tilde z^i_s)-\text{sgn}(z^j_s-\tilde z^j_s)) \nonumber\\
    & = & -2\frac{z^j_s|\tilde z^i_s-z^i_s|+z^i_s|z^j_s-\tilde z^j_s| }{(z^i_s-z^j_s)(\tilde z^i_s-\tilde z^j_s)}|\text{sgn}(z^i_s-\tilde z^i_s)-\text{sgn}(z^j_s-\tilde z^j_s)|\leq 0 \nonumber\\\label{sgncomputation}
  \end{eqnarray}
  
  as the denominator is non-negative.
   Consequently we have for all $M>0$ and $t\geq0$
\begin{equation}
    \sum_{i=1}^n\mathbb{E}|z^i_{t\wedge\tau_M}-\tilde z^i_{t\wedge\tau_M}| \leq -2\gamma\mathbb{E}\int_0^{t\wedge\tau_M}\sum_{i=1}^n|z^i_s-\tilde z^i_s|ds.\nonumber 
\end{equation}

 For $\gamma\geq0$  the left-hand side is $0$. 
 
 For $\gamma<0$, the right-hand side is bounded from above by  $-2\gamma\int_0^t\sum_{i=1}^n\mathbb{E}|z^i_{s\wedge\tau_M}-\tilde z^i_{s\wedge\tau_M}|ds$, and the Grönwall Lemma allows to conclude that for all $M>0$ and $t\geq0$
 \begin{equation}
    \sum_{i=1}^n\mathbb{E}|z^i_{t\wedge\tau_M}-\tilde z^i_{t\wedge\tau_M}| =0.\nonumber 
\end{equation}
  Using Fatou's Lemma to take the limit  $M$ going to infinity, we deduce that for all $t\geq0$
  \begin{equation}
    \sum_{i=1}^n\mathbb{E}|z^i_{t\wedge T\wedge\tilde T}-\tilde z^i_{t\wedge T\wedge\tilde T}| =0 \nonumber 
\end{equation}
which concludes the proof of pathwise uniqueness.
  
  \paragraph{}
 
    Let $Z=(z^1,\dots,z^n)$ and $\tilde Z=(\tilde z^1,\dots ,\tilde z^n)$ be two solutions to (\ref{eds}) defined globally in time, with integrable initial conditions, i.e.
   \[\mathbb{E}\left[\sum_{i=1}^nz^i_0+\tilde z^i_0\right]<+\infty,\]
   and independent from the same driving Brownian motion $\mathbf{B}$. Applying Itô's formula like in (\ref{concentration}),  then an  integration by parts, and last (\ref{sgncomputation}),   we obtain
\begin{align}
    e^{2\gamma t}\left(\sum_{i=1}^n|z^i_t-\tilde z^i_t|\right) & =\sum_{i=1}^n|z^i_0-\tilde z^i_0|+2\sum_{i=1}^n\int_0^te^{2\gamma s}\left|\sqrt{z^i_s}-\sqrt{\tilde z^i_s}\right|dB^i_s+\beta\int_0^{t}e^{2\gamma s}\sum_{i=1}^n\text{sgn}(z^i_s-\tilde z^i_s)\sum_{j\neq i}\left(\frac{z^i_s+z^j_s}{z^i_s-z^j_s}-\frac{\tilde z^i_s+\tilde z^j_s}{\tilde z^i_s-\tilde z^j_s}\right)ds\nonumber\\
    &\leq \sum_{i=1}^n|z^i_0-\tilde z^i_0|+2\sum_{i=1}^n\int_0^te^{2\gamma s}\left|\sqrt{z^i_s}-\sqrt{\tilde z^i_s}\right|dB^i_s.\label{ippp}
\end{align}

  For all $t\geq0$,
  \begin{eqnarray}
    \langle\sum_{i=1}^n\int_0^\cdot\left|\sqrt{z^i_s}-\sqrt{\tilde z^i_s}\right|dB^i_s,\sum_{i=1}^n\int_0^\cdot\left|\sqrt{z^i_s}-\sqrt{\tilde z^i_s}\right|dB^i_s\rangle_t & = & \sum_{i=1}^n\int_0^t\left|\sqrt{z^i_s}-\sqrt{\tilde z^i_s}\right|^2ds\nonumber\\
    &\leq & 2\sum_{i=1}^n\int_0^tz^i_s+\tilde z^i_sds. \nonumber
    \end{eqnarray}
    As the equation (\ref{sumcir}) shows that the sum of the coordinates of $Z$ and $\tilde Z$ are both CIR processes, and by application of Lemma \ref{lamberton} point $4.$, we deduce that
    \begin{eqnarray}
    \mathbb{E}\left[ \langle\sum_{i=1}^n\int_0^\cdot\left|\sqrt{z^i_s}-\sqrt{\tilde z^i_s}\right|dB^i_s,\sum_{i=1}^n\int_0^\cdot\left|\sqrt{z^i_s}-\sqrt{\tilde z^i_s}\right|dB^i_s\rangle_t\right] & \leq &2\mathbb{E}\left[\int_0^t\sum_{i=1}^nz^i_s+\sum_{i=1}^n\tilde z^i_sds\right]<\infty \nonumber
  \end{eqnarray} 
  Consequently, the stochastic integrals in (\ref{ippp}) have zero expectation
 which gives
 \begin{equation}
    \sum_{i=1}^n\mathbb{E}|z^i_{t}-\tilde z^i_{t}| \leq\left(\sum_{i=1}^n\mathbb{E}|z^i_0-\tilde z^i_0|\right)\exp(-2\gamma t) \nonumber
\end{equation}
and ends the proof.
  \end{proof}
  
  For the equation (\ref{eds}) to make sense, integrability of the drift is needed, i.e.
\[\text{for all }i\in\{1,\dots,n\},t\geq0, \int_0^t\left|\alpha-2\gamma\lambda_s^{i}+ \beta\sum_{j\ne i}\frac{\lambda_s^{i}+\lambda_s^{j}}{\lambda_s^{i}-\lambda_s^{j}} \right|ds<\infty \text{ a.s.}.\]
We give a simpler equivalent condition in the next Lemma :
\begin{equation*}
    \text{ for all }t\geq0, \sum_{i=1}^{n-1}\int_0^t\frac{\lambda^{i+1}_s}{\lambda^{i+1}_s-\lambda^{i}_s}ds<\infty \text{ a.s..}\label{24}
\end{equation*}
  
   \begin{lemma}
    \label{4}
Let  $\alpha\geq0,\beta>0$,$\gamma\in\mathbb{R}$, $n\geq2$.    Let $f:\mathbb{R}_+\mapsto \{(x^1,\dots,x^n)\in\mathbb{R}_+, 0\leq x^1\leq\dots\leq x^n\}$  a continuous deterministic  function such as $0\leq f^1_t<f^2_t<\dots<f^n_t$ dt-a.e..
    
   Then we have for all $t\geq0$
   \begin{eqnarray}
   \label{edscond}
    \sum_{i=1}^{n-1}\int_0^t\frac{f_s^{i+1}}{f_s^{i+1}-f_s^{i}}ds<+\infty \Leftrightarrow  \sum_{i=1}^n\int_0^t\left|\alpha-2\gamma f^i_s+\beta\sum_{j\neq i}\frac{f^i_s+f^j_s}{f^i_s-f^j_s}\right|ds<+\infty 
\end{eqnarray}
where by convention for $i\neq j$, $\frac{f^i_s}{f^j_s-f^i_s}=+\infty$ when $f^j_s=f^i_s$.
  \end{lemma}
  \begin{proof}

  The proof of the direct implication of (\ref{edscond}) is straightforward.
  
To prove the converse implication, let us check by backward induction on $i\in\{1,\dots, n\}$ that \[ \text{ for all }1\leq l<i\leq n,t\geq0, \int_0^t\frac{f_s^i}{f_s^{i}-f_s^{l}}ds<+\infty,\] which implies condition (\ref{edscond}).
  For all $i\in\{1,\dots,n\}$, $t\geq0$ we have

  \[ \int_0^t\left|\alpha-2\gamma f^i_s+\beta\sum_{j\neq i}\frac{f^i_s+f^j_s}{f^i_s-f^j_s}\right|ds= \int_0^t\left|\alpha-(n-1)\beta-2\gamma f_s^{i}+ 2\beta f_s^i\sum_{j\ne i}\frac{1}{f_s^{i}-f_s^{j}}\right|ds<+\infty\] and thus as $f$ is continuous , 
  \[ \text{ for all }t\geq0, \int_0^t \left|f_s^i\sum_{j\ne i}\frac{1}{f_s^{i}-f_s^{j}}\right|ds<+\infty.\]
  For $i=n$ we have as the coordinates are ordered, 
   \[ \text{ for all }t\geq0, \int_0^t\left|f_s^n\sum_{j= 1}^{n-1}\frac{1}{f_s^{n}-f_s^{j}}\right|ds=\sum_{j= 1}^{n-1} \int_0^t\frac{f_s^n}{f_s^{n}-f_s^{j}}<+\infty\] which gives the result. 
   
   Let $1\leq i<n$ and let us assume the induction hypothesis for all $k\in\{i+1,\dots, n\}$. We have 
   \begin{align}
   \label{firsteq}
       \text{ for all }t\geq0, \int_0^t\left|-f_s^{i}\sum_{j> i}\frac{1}{f_s^{j}-f_s^{i}}+f_s^{i}\sum_{j< i}\frac{1}{f_s^{i}-f_s^{j}}\right|ds&<+\infty\nonumber\\
      \Leftrightarrow \text{ for all }t\geq0, \int_0^t\left|(n-i)-\sum_{j> i}\frac{f_s^{j}}{f_s^{j}-f_s^{i}}+f_s^{i}\sum_{j< i}\frac{1}{f_s^{i}-f_s^{j}}\right|ds&<+\infty\\
       \Rightarrow \text{ for all }t\geq0, \int_0^t f_s^{i}\sum_{j< i}\frac{1}{f_s^{i}-f_s^{j}}ds&<+\infty\nonumber
   \end{align}
as the first sum in (\ref{firsteq}) is integrable by the induction assumption and the terms in the second sum are all non-negative. This gives the result and allows to conclude.

  \end{proof}
  
    \begin{proof}[Proof of Proposition \ref{stationary}]
    
  Let us then show that  an invariant distribution $\rho_{inv}$ has a finite first order moment. To do so, one can remark (see (\ref{sumcir})) that the image by the sum of the $n$ coordinates of $\rho_{inv}$ is invariant for the CIR process \[dr_t=2\sqrt{r_t}dW_t+(n\alpha -2\gamma r_t)dt.\] It is known (see for instance \cite{RePEc:ecm:emetrp:v:53:y:1985:i:2:p:385-407}) that the invariant distribution of such a process is a gamma law of positive parameters, whose density is
  \[r\mapsto \frac{\gamma^{\frac{n\alpha}{2}}}{\Gamma(\frac{n\alpha}{2})}r^{\frac{n\alpha}{2}-1}e^{-\gamma r},\]
  which has a finite first order moment.  We can thus first apply the second part of Lemma \ref{pathwise_uniqueness} for two solutions to (\ref{eds}) starting respectively according to two invariant distributions to deduce that these two invariant distributions are equal.

  \paragraph{}
  Let us now  exhibit the invariant distribution $\rho_{inv}$ of the process. It should solve the Fokker-Planck equation 
  \begin{equation}
      \mathcal{A}^*\rho_{inv} = 0 \label{fokker_planck}
  \end{equation}
  where $\mathcal{A}$ is the infinitesimal generator of the diffusion $\Lambda$ : 
  \begin{align*}
  \mathcal{A} &= \sum_{i=1}^nb_i(\lambda^1,\dots,\lambda^n)\frac{\partial}{\partial \lambda^i}+2\sum_{i=1}^n\lambda^i\frac{\partial^2}{\partial (\lambda^i)^2}
  \end{align*}
  where $b_i(\lambda^1,\dots,\lambda^n) = \alpha-2\gamma\lambda^i+\beta\sum_{j\neq i}\frac{\lambda^i+\lambda^j}{\lambda^i-\lambda^j}$.
  The candidate to be the stationary distribution of the gradient diffusion process defined in (\ref{sqrteds}) has the density \[f(x^1,\dots,x^n)= \frac{1}{\mathcal{Z}}\exp\left\{-2V\left(x^1, \dots, x^n\right)\right\}\mathds{1}_{0\leq x^1\leq\dots\leq x^n}\]
  where $V$ is defined in (\ref{potentiam}) by
  \begin{equation*}
  V(x_1,...,x_n)=- \sum_{i=1}^n\left\{\frac{\alpha-1-(n-1)\beta}{2}\ln| x^{i}|-\frac{1}{2}\gamma (x^{i})^{2}+\frac{\beta}{4}\sum_{j\ne i}\left(\ln| x^{i}-x^{j}| + \ln| x^{i}+x^{j}|\right)\right\}.
 \end{equation*} and \[\mathcal{Z} = \int_{0\leq x^1\leq\dots\leq x^n}\exp\left\{-2V\left(x^1, \dots, x^n\right)\right\}dx^1\dots dx^n.\]
   By a square root change of variables, the natural candidate to be the density of the stationary distribution $\rho_{inv}$ of the process defined in (\ref{eds}) is
   \begin{align*}
       f_{inv}(\lambda^1,\dots,\lambda^n) &= \frac{1}{2^n}\times\frac{1}{\sqrt{\lambda^1\times\dots\times\lambda^n}}\times\frac{1}{\mathcal{Z}}\times\exp\left\{-2V\left(\sqrt{\lambda^1}, \dots, \sqrt{\lambda^n}\right)\right\}\mathds{1}_{0\leq\lambda^1\leq\dots\leq\lambda^n}\\
       &   = \frac{1}{2^n}\times\frac{1}{\sqrt{\lambda^1\times\dots\times\lambda^n}}\times\frac{1}{\mathcal{Z}}\times\prod_{i=1}^n\left((\lambda^i)^{\frac{\alpha-1-(n-1)\beta}{2}}e^{-\frac{\gamma}{2}\lambda^i}\prod_{j\neq i}|\lambda^j-\lambda^i|^{\beta/2}\right)\mathds{1}_{0\leq\lambda^1\leq\dots\leq\lambda^n}.
   \end{align*}
  
  This function is well-defined and integrable. Indeed, for $0<\lambda^1<\dots<\lambda^n$
  \begin{align*}
      \left|f_{inv}(\lambda^1,\dots,\lambda^n)\right|&\leq \frac{1}{2^n}\times\frac{1}{\mathcal{Z}}\times\prod_{i=1}^n\left((\lambda^i)^{\frac{\alpha-2-(n-1)\beta}{2}}e^{-\frac{\gamma}{2}\lambda^i}\prod_{j\neq i}(\lambda^n)^{\beta/2}\right)\\
      &\leq \frac{1}{2^n}\times\frac{1}{\mathcal{Z}}\times(\lambda^n)^{n(n-1)\beta/2}\times\prod_{i=1}^n\left((\lambda^i)^{\frac{\alpha-2-(n-1)\beta}{2}}e^{-\frac{\gamma}{2}\lambda^i}\right)\\
      &\leq \frac{1}{2^n}\times\frac{1}{\mathcal{Z}}\times\prod_{i=1}^{n-1}\left((\lambda^i)^{\frac{\alpha-2-(n-1)\beta}{2}}e^{-\frac{\gamma}{2}\lambda^i}\right)(\lambda^n)^{\frac{(n-1)^2\beta+\alpha-2}{2}}e^{-\frac{\gamma}{2}\lambda^n}
  \end{align*}
  and \[\alpha-(n-1)\beta>0\implies\frac{\alpha-2-(n-1)\beta}{2}>-1. \]
  
  Let us check that $\rho_{inv}$ solves equation (\ref{fokker_planck}) in the sense of distributions. For a test function $\phi$, compactly supported and twice continuously differentiable, since $f_{inv}$ vanishes for $\lambda^1=0$, $\lambda^i=\lambda^{i+1}$ when $i\in\{1,\dots,n-1\}$, and for $\lambda^n\rightarrow+\infty$, we obtain by integration by parts that for $i\in\{1,\dots,n\}$
  \begin{align}
      \int_{0\leq \lambda^1\leq\dots\leq \lambda^n}&\lambda^i\frac{\partial^2\phi}{\partial (\lambda^i)^2}(\lambda^1,\dots,\lambda^n)f_{inv}(\lambda^1,\dots,\lambda^n)d\lambda^1\dots d\lambda^n \nonumber\\
       & = -\int_{0\leq \lambda^1\leq\dots\leq \lambda^n}\frac{\partial\phi}{\partial \lambda^i}(\lambda^1,\dots,\lambda^n)\left(f_{inv}(\lambda^1,\dots,\lambda^n)+\lambda^i\frac{\partial f_{inv}}{\partial \lambda^i}(\lambda^1,\dots,\lambda^n)\right)d\lambda^1\dots d\lambda^n. \label{ipp3}
  \end{align}

  Then we have 
   \begin{align}
    &\int_{0\leq \lambda^1\leq\dots\leq \lambda^n}\mathcal{A}\phi(\lambda^1,\dots,\lambda^n)d\rho_{inv}(\lambda^1,\dots,\lambda^n) \nonumber\\
    &=\sum_{i=1}^n\int_{0\leq \lambda^1\leq\dots\leq \lambda^n}\left[(b_i(\lambda^1,\dots,\lambda^n)-2)f_{inv}(\lambda^1,\dots,\lambda^n)-2\lambda^i\frac{\partial f_{inv}}{\partial \lambda^i}(\lambda^1,\dots,\lambda^n)\right]\frac{\partial \phi}{\partial \lambda^i}(\lambda^1,\dots,\lambda^n)d\lambda^1\dots d\lambda^n.\nonumber
  \end{align}
  As for $0<\lambda^1<\dots<\lambda^n$
  \begin{align}
    \frac{\partial f_{inv} }{\partial\lambda^i}(\lambda^1,\dots,\lambda^n) & = \frac{1}{2^n\mathcal{Z}\sqrt{\lambda^1\times\dots\times\lambda^n}}\exp\left\{-2V\left(\sqrt{\lambda^1}, \dots, \sqrt{\lambda^n}\right)\right\}\left[-\frac{1}{2\lambda^i}-\frac{1}{\sqrt{\lambda^i}}\frac{\partial V}{\partial x^i}(\sqrt{\lambda^1},\dots,\sqrt{\lambda^n})\right]\nonumber\\
    & = -\frac{1}{2\sqrt{\lambda^i}}\left[\frac{1}{\sqrt{\lambda^i}}+2\frac{\partial V}{\partial x^i}(\sqrt{\lambda^1},\dots,\sqrt{\lambda^n})\right]f_{inv}(\lambda^1,\dots,\lambda^n),\nonumber
  \end{align}
  a calculus gives us
  \begin{align}
      (b_i(&\lambda^1,\dots,\lambda^n)-2)f_{inv}(\lambda^1,\dots,\lambda^n)-2\lambda^i\frac{\partial f_{inv}}{\partial \lambda^i}(\lambda^1,\dots,\lambda^n)\nonumber\\ &=f_{inv}(\lambda^1,\dots,\lambda^n)\left[b_i(\lambda^1,\dots,\lambda^n)-2+\sqrt{\lambda^i}\left(\frac{1}{\sqrt{\lambda^i}}+2\frac{\partial V}{\partial x^i}(\sqrt{\lambda^1},\dots,\sqrt{\lambda^n})\right)\right] \nonumber\\
      & = f_{inv}(\lambda^1,\dots,\lambda^n)\left[\alpha-2\gamma\lambda^i+\beta\sum_{j\neq i}\frac{\lambda^i+\lambda^j}{\lambda^i-\lambda^j}-2+1+2\sqrt{\lambda^i}\left(-\frac{\alpha-1}{2}\frac{1}{\sqrt{\lambda^i}}+\gamma\sqrt{\lambda^i}-\frac{\beta}{2\sqrt{\lambda^i}}\sum_{j\neq i}\frac{\lambda^i+\lambda^j}{\lambda^i-\lambda^j}\right)\right]\nonumber\\
      & = 0.\nonumber
  \end{align}
  The distribution $\rho_{inv}$ is thus solving (\ref{fokker_planck}) in the sense of distributions. Let us show that it is an invariant distribution for the process defined by (\ref{eds}) using \cite[Theorem 2.5]{MR3485364}. To do so, we only need to verify that 
  \[\int_{0\leq \lambda^1\leq\dots\leq \lambda^n}2\sum_{i=1}^n\left\{\lambda^i+\left|b_i(\lambda_1,\dots,\lambda^n)\right|\right\}d\rho_{inv}(\lambda^1,\dots,\lambda^n)<+\infty.\]
We have 
\begin{align*}
    \int_{0\leq \lambda^1\leq\dots\leq \lambda^n}&2\sum_{i=1}^n\left\{\lambda^i+\left|b_i(\lambda_1,\dots,\lambda^n)\right|\right\}f_{inv}(\lambda^1,\dots,\lambda^n)d\lambda^1\dots d\lambda^n\\
    & = \int_{0\leq \lambda^1\leq\dots\leq \lambda^n}2\sum_{i=1}^n\left\{\lambda^i+\left|\alpha-2\lambda^i+\beta\sum_{j\neq i}\frac{\lambda^i+\lambda^j}{\lambda^i-\lambda^j}\right|\right\}f_{inv}(\lambda^1,\dots,\lambda^n)d\lambda^1\dots d\lambda^n.
\end{align*}
The exponential factors in $f_{inv}$ crushes every other term on the $+\infty$ boundary. If $\beta\geq1$, the boundary $\lambda^i=\lambda^j$ is not singular. Let us discuss further this boundary when $\beta<1$ by looking at the term below for $1\leq j<i\leq n$ :
\begin{align}
   \int_{0\leq \lambda^1\leq\dots\leq \lambda^n}&\frac{\lambda^i+\lambda^j}{\lambda^i-\lambda^j}f_{inv}(\lambda^1,\dots,\lambda^n)d\lambda^1\dots d\lambda^n \nonumber\\
   &\leq \frac{1}{2^{n-1}}\times\frac{1}{\mathcal{Z}}\int_{0\leq \lambda^1\leq\dots\leq \lambda^n}(\lambda^i-\lambda^j)^{-1}\lambda^n\prod_{k=1}^n\left((\lambda^k)^{\frac{\alpha-2-(n-1)\beta}{2}}e^{-\frac{\gamma}{2}\lambda^k}\prod_{l\neq k}|\lambda^l-\lambda^k|^{\beta/2}\right)d\lambda^1\dots d\lambda^n\label{1toexpl}\\
    &\leq \frac{1}{2^{n-1}}\times\frac{1}{\mathcal{Z}}\int_{0\leq \lambda^1\leq\dots\leq \lambda^n}(\lambda^i-\lambda^j)^{-1+\beta}(\lambda^n)^{1-\beta+\frac{n(\alpha-2)}{2}}e^{-\frac{\gamma}{2}\lambda^n}d\lambda^1\dots d\lambda^n\label{2toexpl}\\
    &\leq \frac{1}{2^{n-1}}\times\frac{1}{\mathcal{Z}}\int_{(\mathbb{R}_+)^{n-1}}\left(\int_{\mathbb{R}_+}\mathds{1}_{0\leq\lambda^1\leq\lambda^2}d\lambda^1\right)(\lambda^i-\lambda^j)^{-1+\beta}(\lambda^n)^{1-\beta+\frac{n(\alpha-2)}{2}}e^{-\frac{\gamma}{2}\lambda^n}\mathds{1}_{0\leq\lambda^2\leq\dots\leq\lambda^n}d\lambda^2\dots d\lambda^n\nonumber\\
    &\leq \frac{1}{2^{n-1}}\times\frac{1}{\mathcal{Z}}\int_{(\mathbb{R}_+)^{n-1}}(\lambda^i-\lambda^j)^{-1+\beta}(\lambda^n)^{2-\beta+\frac{n(\alpha-2)}{2}}e^{-\frac{\gamma}{2}\lambda^n}\mathds{1}_{0\leq\lambda^2\leq\dots\leq\lambda^n}d\lambda^2\dots d\lambda^n\nonumber\\
    &\text{ }\vdots\nonumber\\
    &\leq \frac{1}{2^{n-1}}\times\frac{1}{\mathcal{Z}}\int_{(\mathbb{R}_+)^{n-j+1}}(\lambda^i-\lambda^j)^{-1+\beta}(\lambda^n)^{j-\beta+\frac{n(\alpha-2)}{2}}e^{-\frac{\gamma}{2}\lambda^n}\mathds{1}_{0\leq\lambda^j\leq\dots\leq\lambda^n}d\lambda^j\dots d\lambda^n\nonumber\\
    &\leq \frac{1}{2^{n-1}}\times\frac{1}{\mathcal{Z}}\int_{(\mathbb{R}_+)^{n-j}}\left(\int_{\mathbb{R}_+}(\lambda^i-\lambda^j)^{-1+\beta}\mathds{1}_{0\leq\lambda^j\leq\lambda^{j+1}}d\lambda^j\right)(\lambda^n)^{j-\beta+\frac{n(\alpha-2)}{2}}e^{-\frac{\gamma}{2}\lambda^n}\mathds{1}_{0\leq\lambda^{j+1}\leq\dots\leq\lambda^n}d\lambda^{j+1}\dots d\lambda^n\nonumber\\
    &\leq \frac{1}{2^{n-1}(1-\beta)}\times\frac{1}{\mathcal{Z}}\int_{(\mathbb{R}_+)^{n-j}}(\lambda^n)^{j+\frac{n(\alpha-2)}{2}}e^{-\frac{\gamma}{2}\lambda^n}\mathds{1}_{0\leq\lambda^{j+1}\leq\dots\leq\lambda^n}d\lambda^{j+1}\dots d\lambda^n<+\infty\nonumber\\
    &\text{ }\vdots\nonumber\\
    &\leq \frac{1}{2^{n-1}(1-\beta)}\times\frac{1}{\mathcal{Z}}\int_{\mathbb{R}_+}(\lambda^n)^{-1+\frac{n\alpha}{2}}e^{-\frac{\gamma}{2}\lambda^n}d\lambda^n<+\infty\nonumber
\end{align}
as $\alpha>(n-1)\beta>0$. To go from (\ref{1toexpl}) to (\ref{2toexpl}), we  bound from above $(\lambda^k)^{\frac{\alpha-2-(n-1)\beta}{2}}$ by $(\lambda^n)^{\frac{\alpha-2-(n-1)\beta}{2}}$, $|\lambda^l-\lambda^k|^{\beta/2}$ for $\{l,k\}\neq\{i,j\}$ by $(\lambda^n)^\frac{\beta}{2}$ and $e^{-\frac{\gamma}{2}\lambda^k}$ for $k\neq n$ by $1$.  We can thus apply \cite[Theorem 2.5]{MR3485364} to deduce that we can define a process solving (\ref{eds}) whose marginals follow the law $\rho_{inv}$.

The distribution $\rho_{inv}$ is thus invariant, allows to build a weak  solution to (\ref{eds}) (by taking every marginal distributed according to $\rho_{inv}$). By pathwise uniqueness, it is a strong solution.

  \end{proof}
  
  \section{Proof of Theorem \ref{existence}}

We start this section by the proof of Proposition \ref{multiplecollision} since this result is crucial in the proof of Theorem \ref{existence}.

\paragraph{Proof of Proposition \ref{multiplecollision}}
\paragraph{(i)}
To prove this assertion, we study for all $k\in\{1,\dots,n\}$ the process $\lambda^1+\dots+\lambda^k$ and show that, as its interaction terms  with the particles $\lambda^{k+1},\dots,\lambda^n$ are non-positive, it is smaller than a CIR process hitting zero in finite time.

Let us define $W^k$ by $W^k_0=0$ and \\ $dW^k_t=\sum_{i=1}^k\left(\mathds{1}_{\left\{\sum_{j=1}^k\lambda^j_t\neq0\right\}}\frac{\sqrt{\lambda^i_t}}{\sqrt{\sum_{j=1}^k\lambda^j_t}}+\mathds{1}_{\left\{\sum_{j=1}^k\lambda^j_t=0\right\}}\frac{1}{\sqrt{k}}\right)dB^i_t$.\\  According to Lévy's characterization, $W^k$ is a Brownian motion. For $t\geq0$,
  \begin{eqnarray}
  d(\lambda^{1}_t+\dots+\lambda^k_t) & = & 2\sqrt{\lambda^{1}_t+\dots+\lambda^k_t}dW^k_t-2\gamma(\lambda^{1}_t+\dots+\lambda^k_t)dt+k(\alpha-(n-k)\beta)dt+2\beta\sum_{i=1}^k\lambda^i_t\sum_{j=k+1}^n\frac{1}{\lambda^i_t-\lambda_t^j}dt\nonumber \\
  &\leq& 2\sqrt{\lambda^{1}_t+\dots+\lambda^k_t}dW^k_t-2\gamma(\lambda^{1}_t+\dots+\lambda^k_t)dt+k(\alpha-(n-k)\beta)dt.\label{aaaaa}
  \end{eqnarray}
By the pathwise comparison theorem of Ikeda and Watanabe (that we recall in Theorem \ref{Ikeda} below),
$$\lambda^1_t+\dots+\lambda^k_t\leq r_t \text{ for all } t\geq0 \text{ a.s.}$$ where  
\begin{equation}
\label{r}
    r_t= \lambda^1_0+\dots+\lambda^k_0+2\int_0^t\sqrt{r_s}dW^k_s-2\gamma \int_0^tr_sds+k(\alpha-(n-k)\beta)t
\end{equation}

 is a CIR process. Applying  Lemma \ref{lamberton} with $a=k(\alpha-(n-k)\beta)$, $b=2\gamma$ and $\sigma=2$ which satisfy $a<\frac{\sigma^2}{2}$ and $b\geq0$, we can conclude.

\paragraph{(ii)}
To prove this assertion, we proceed by backward induction on $k$. Indeed, if $k(\alpha-(n-k)\beta)\geq2$ then for all $l\geq k$, $
l(\alpha-(n-l)\beta)\geq2$. The idea is to show that the process $\lambda^1+\dots+\lambda^k$ is bigger than a CIR process which never hits zero. To do so, we exploit the fact that this process cannot hit zero at the same time as the coordinate $\lambda^{k+1}$  by induction assumption. 

For all $k\in\{1,\dots,n\}$, we define the Brownian motion $W^k$ the following way :\\ $dW^k_t=\mathds{1}_{\{t\in[0,\mathcal{T})\}}\sum_{i=1}^k\left(\mathds{1}_{\left\{\sum_{j=1}^k\lambda^j_t\neq0\right\}}\frac{\sqrt{\lambda^i_t}}{\sqrt{\sum_{j=1}^k\lambda^j_t}}+\mathds{1}_{\left\{\sum_{j=1}^k\lambda^j_t=0\right\}}\frac{1}{\sqrt{k}}\right)dB^i_t+\mathds{1}_{\{t\geq\mathcal{T}\}}\frac{1}{\sqrt{k}}\sum_{i=1}^kdB^i_t$.
\\For $k=n$ the inequality (\ref{aaaaa}) is an equality for $t<\mathcal{T}$, and according to Lemma \ref{lamberton}, the CIR process $r$ defined by (\ref{r}) is defined  globally in time and for $t\in[0,\mathcal{T})$ we have $$\lambda^1_t+\dots+\lambda^n_t = r_t >0$$ and $r_\mathcal{T}>0$ on $\{\mathcal{T}<\infty\}$ a.s..  We can thus conclude. 

Let us now assume that for some $k\in\{1,\dots,n\}$, $k(\alpha-(n-k)\beta)\geq2$ and $$\mathbb{P}\left(\mathcal{T}<\infty,\underset{t\in[0,\mathcal{T})}{\inf} (\lambda_t^{1}+\lambda_t^{2}+\dots+\lambda_t^{k+1})=0\right)+\mathbb{P}\left(\mathcal{T}=\infty,\exists t\geq0 : \lambda_t^{1}+\lambda_t^{2}+\dots+\lambda_t^{k+1}=0\right)=0. $$ Then, 

\begin{eqnarray}
    \mathbb{P}\Big(\mathcal{T}<\infty,\underset{t\in[0,\mathcal{T})}{\inf}(\lambda^{1}_t+\dots+\lambda^{k}_t)=0\Big)
    = \underset{\epsilon\downarrow0}\lim\mathbb{P}\Bigg(\mathcal{T}<\infty,\underset{t\in[0,\mathcal{T})}{\inf} (\lambda_t^{1}+\dots+\lambda_t^{k}+\mathds{1}_{\{\lambda_t^{k+1}-\lambda_t^k<\epsilon\}})=0\Bigg)  \nonumber
\end{eqnarray}
and
\begin{eqnarray}
    \mathbb{P}\left(\mathcal{T}=\infty,\exists t\geq0 : \lambda_t^{1}+\lambda_t^{2}+\dots+\lambda_t^{k}=0\right)&=&\underset{\epsilon\downarrow0}\lim\mathbb{P}\Bigg(\mathcal{T}=\infty,\exists t\geq0 : \lambda_t^{1}+\lambda_t^{2}+\dots+\lambda_t^{k} = 0 \text{ and }\lambda_t^{k+1}-\lambda_t^k\geq\epsilon\Bigg)  \nonumber\\
    &=&\underset{\epsilon\downarrow0}\lim\mathbb{P}\Bigg(\mathcal{T}=\infty,\exists t\geq0 : \lambda_t^{1}+\lambda_t^{2}+\dots+\lambda_t^{k}+\mathds{1}_{\{\lambda_t^{k+1}-\lambda_t^k<\epsilon\}}=0\Bigg).  \nonumber
\end{eqnarray}
For $\epsilon>0$, starting with
\begin{eqnarray}
\tau^0_\epsilon &=&\inf\{t \in[0,\mathcal{T}): \lambda_t^{k+1}-\lambda_t^k\geq\epsilon\}, \nonumber
\end{eqnarray}
let us define inductively for $j\in\mathbb{N}$
\begin{eqnarray}
\sigma_\epsilon^j & = & \inf\{t\in(\tau_\epsilon^j,\mathcal{T}) : \lambda_t^{k+1}-\lambda_t^k\leq\frac{\epsilon}{2}\} , \nonumber\\
\tau_\epsilon^{j+1}&=&\inf\{t\in(\sigma_\epsilon^j,\mathcal{T}): \lambda_t^{k+1}-\lambda_t^k\geq\epsilon\}, \nonumber
\end{eqnarray}
with the convention $\inf\emptyset=\mathcal{T}$. As the function $t\rightarrow\lambda_t^{k+1}-\lambda_t^k$ is continuous on $[0,\mathcal{T})$, 
\[\sigma_\epsilon^j,\tau_\epsilon^j\underset{j\rightarrow +\infty}{\longrightarrow}\mathcal{T}.\]

As  $\lambda_t^{k+1}-\lambda_t^{k}<\epsilon$ on $[0,\tau^0_\epsilon)$ and $[\sigma^j_\epsilon,\tau^{j+1}_\epsilon)$ for $j\in\mathbb{N}$,
\begin{eqnarray}
    \Bigg\{\mathcal{T}<\infty &,& \underset{t\in[0,\mathcal{T})}{\inf} (\lambda_t^{1}+\dots+\lambda_t^{k}+\mathds{1}_{\{\lambda_t^{k+1}-\lambda_t^k<\epsilon\}})=0\Bigg\} =\left\{\mathcal{T}<\infty,\exists j\in\mathbb{N}^* ,\underset{t\in[\tau_\epsilon^j, \sigma_\epsilon^j)}{\inf}(\lambda_t^{1}+\dots+\lambda_t^{k})=0 \right\}\nonumber\\
   \text{ and } \Bigg\{\mathcal{T}=\infty&,&\exists t\geq0 : \lambda_t^{1}+\lambda_t^{2}+\dots+\lambda_t^{k}+\mathds{1}_{\{\lambda_t^{k+1}-\lambda_t^k<\epsilon\}} = 0\Bigg\} \nonumber\\ &&=\Bigg\{\mathcal{T}=\infty,\exists j\in\mathbb{N}^* ,\exists t\in[\tau_\epsilon^j, \sigma_\epsilon^j) : \lambda_t^{1}+\lambda_t^{2}+\dots+\lambda_t^{k} = 0 \Bigg\},\nonumber
\end{eqnarray}

it is enough to check that
\begin{eqnarray}
    &&\mathbb{P}\left(\left\{\mathcal{T}<+\infty, \exists j\in\mathbb{N}^* ,\underset{t\in[\tau_\epsilon^j, \sigma_\epsilon^j)}{\inf}(\lambda_t^{1}+\dots+\lambda_t^{k})=0 \right\}\right)\nonumber\\
    &+& \mathbb{P}\left(\Bigg\{\mathcal{T}=+\infty, \exists j\in\mathbb{N}^* ,\exists t\in[\tau_\epsilon^j, \sigma_\epsilon^j) : \lambda_t^{1}+\lambda_t^{2}+\dots+\lambda_t^{k} = 0 \Bigg\}\right) = 0. \nonumber
\end{eqnarray}
We have for  all $t\in[\tau_\epsilon^j,\sigma_\epsilon^j)$, $\lambda_t^{k+1}-\lambda_t^{k}>\frac{\epsilon}{2}$ and
\begin{eqnarray}
  d(\lambda^{1}_t+\dots+\lambda^k_t)  &=&  2\sqrt{\lambda^{1}_t+\dots+\lambda^k_t}dW^k_t-\gamma(\lambda^{1}_t+\dots+\lambda^k_t)dt+k(\alpha-(n-k)\beta)dt+2\beta\sum_{i=1}^k\lambda^i_t\sum_{j=k+1}^n\frac{1}{\lambda^i_t-\lambda_t^j}dt\nonumber \\
  &\geq&  2\sqrt{\lambda^{1}_t+\dots+\lambda^k_t}dW^k_t-\gamma(\lambda^{1}_t+\dots+\lambda^k_t)dt+k(\alpha-(n-k)\beta)dt-\frac{4}{\epsilon}\beta(n-k)\sum_{i=1}^k\lambda_t^i \nonumber\\
  &\geq&  2\sqrt{\lambda^{1}_t+\dots+\lambda^k_t}dW^k_t-\left(\gamma+\frac{4}{\epsilon}\beta(n-k)\right)(\lambda^{1}_t+\dots+\lambda^k_t)dt+k(\alpha-(n-k)\beta)dt. \label{kbiggerthancir}
  \end{eqnarray}

  We can then define on $\{\tau^j_\epsilon<\infty\}$ the process $ r^j$   by $r^j_0 =\left(\lambda^{1}_{\tau_\epsilon^j}+\dots+\lambda^k_{\tau_\epsilon^j}\right)\mathds{1}_{\{\tau^j_\epsilon<\mathcal{T}\}}+\mathds{1}_{\{\tau^j_\epsilon=\mathcal{T}\}} $ and  for all $t\geq0$ :
  \begin{eqnarray}
    d r^j_t &=&2\sqrt{ r^j_t}dW^k_{t+\tau_\epsilon^j}+\left[-\left(\gamma+\frac{4}{\epsilon}\beta(n-k)\right) r^j_t+k(\alpha-(n-k)\beta)\right]dt \nonumber\\
     &=&2\sqrt{ r^j_t}d(W^k_{t+\tau_\epsilon^j}-W^k_{\tau_\epsilon^j}+W^k_{\tau_\epsilon^j})+\left[-\left(\gamma+\frac{4}{\epsilon}\beta(n-k)\right) r^j_t+k(\alpha-(n-k)\beta)\right]dt \nonumber\\
     &=&2\sqrt{ r^j_t}d\tilde W^k_{t}+\left[-\left(\gamma+\frac{4}{\epsilon}\beta(n-k)\right) r^j_t+k(\alpha-(n-k)\beta)\right]dt \nonumber
  \end{eqnarray}
  where conditionally on $\{\tau^j_\epsilon<\infty\}$, by strong Markov property, $(\tilde W^k_t = W^k_{t+\tau_\epsilon^j}-W^k_{\tau_\epsilon^j})_{t\geq0}$ is a Brownian motion independent from $\mathcal{F}_{\tau_\epsilon^j}$.
Conditionally on $\{\tau^j_\epsilon<\infty\}$, the process  $ r^j$ is a CIR process defined globally in time  according to  Lemma \ref{lamberton} with $a=k(\alpha-(n-k)\beta)$ and $\sigma=2$ which satisfy $a\geq\frac{\sigma^2}{2}$, and it stays positive on $\mathbb{R}_+$.

 This together with (\ref{kbiggerthancir}) and Theorem \ref{Ikeda} give that for all $t\in[\tau_\epsilon^j, \sigma_\epsilon^j)$,
  \begin{equation}
      \lambda^{1}_t+\dots+\lambda^k_t\geq r_{t-\tau^j_\epsilon}^j. \nonumber
  \end{equation}
  
   We can thus conclude. 

\paragraph{Proof of Theorem \ref{existence}}

  As explained in the introduction, the main difficulty in proving this result comes from the fact that we have to deal with both singularities when a particle hits zero and when two particles collide at the same time. For $\epsilon>0$, our method precisely consists in separating these difficulties by defining two new SDEs (\ref{hatAeta}) and ($B_\epsilon$) which each remove one type of singularity and coincide with (\ref{eds}) on domains that cover $\{t\geq0: 0\leq\lambda^1_t\leq\dots\leq\lambda^n_t\text{ and } (\lambda^1_t\geq\epsilon\text{ or }\lambda^2_t-\lambda^1_t\geq\epsilon)\}$. This allows us to build a solution to (\ref{eds}) by piecing together solutions to (\ref{hatAeta}) and ($B_\epsilon$).
  
  Let us consider in this proof the Brownian motion $\mathbf{B}=(B^1_t,\dots,B^n_t)_{t\geq0}$, $\mathcal{F}_t=\sigma\left((\lambda^1_0,\dots,\lambda^n_0), \left(\mathbf{B}_s \right)_{s \le t} \right)$ for all $t\geq0$, and the SDE defined by (\ref{eds}).
  Let us define for all $\epsilon>0$ the following SDEs :
  
  \begin{equation}
  \label{hatAeta}
         d \hat\lambda^{i,\epsilon}_t = 2\sqrt{\hat\lambda^{i,\epsilon}_t}dB^i_t+\left[\alpha-(n-1)\beta+1-0\vee\frac{2\sqrt{2}}{\sqrt{\epsilon}}\left(\sqrt{\hat\lambda^{i,\epsilon}_t}-\frac{\sqrt{\epsilon}}{2\sqrt{2}}\right)\wedge1 -2\gamma\hat\lambda^{i,\epsilon}_t+2\beta\hat\lambda^{i,\epsilon}_t\sum_{j\neq i}\frac{1}{\hat\lambda^{i,\epsilon}_t-\hat\lambda^{j,\epsilon}_t}\right]dt \tag{$\hat{A} _\epsilon$}
  \end{equation}
  \begin{equation*}
      0\leq \hat\lambda^{1,\epsilon}_t<\dots< \hat\lambda^{n,\epsilon}_t, \text{ a.s.,  } dt-a.e..
  \end{equation*}

  \begin{equation}
      \label{Beta}
\left\{\begin{aligned}
    d\tilde\lambda^{1,\epsilon}_t &= 2\sqrt{\tilde\lambda^{1,\epsilon}_t}dB^1_t+(\alpha-(n-1)\beta)dt-2\gamma\tilde\lambda^{1,\epsilon}_tdt-2\beta\sum_{j\neq1}\frac{\tilde\lambda^{1,\epsilon}_t\wedge\epsilon}{(\tilde\lambda^{j,\epsilon}_t-\tilde\lambda^{1,\epsilon}_t\wedge\epsilon)\vee\epsilon}dt \\
\forall i\in\{2,\dots, n \}\text{, }  d \tilde\lambda^{i,\epsilon}_t &= 2\sqrt{\tilde\lambda^{i,\epsilon}_t}dB^i_t+(\alpha-(n-1)\beta+1)dt -2\gamma\tilde\lambda^{i,\epsilon}_t+2\beta\sum_{j\geq 2, j\neq i}\frac{\tilde\lambda^{i,\epsilon}_t}{\tilde\lambda^{i,\epsilon}_t-\tilde\lambda^{j,\epsilon}_t}dt\nonumber\\
   &-\left[0\vee\frac{2}{\sqrt{\epsilon}}\left(\sqrt{\tilde\lambda^{i,\epsilon}_t}-\frac{\sqrt{\epsilon}}{2}\right)\wedge1\right]dt+2\beta\frac{\tilde\lambda^{i,\epsilon}_t}{(\tilde\lambda^{i,\epsilon}_t-\tilde\lambda^{1,\epsilon}_t\wedge\epsilon)\vee\epsilon}dt. \nonumber
\end{aligned}\right.
\tag{$B_\epsilon$}
\end{equation}
\begin{equation*}
      0\leq \tilde\lambda^{1,\epsilon}_t \text{ and } 0\leq \tilde\lambda^{2,\epsilon}_t<\dots< \tilde\lambda^{n,\epsilon}_t, \text{ a.s.,  } dt-a.e..
  \end{equation*}

  These systems are built such as : 
  \begin{eqnarray}
    (\hat A_\epsilon) &\text{  coincides with (\ref{eds}) on } &\{t, \hat\lambda^{1,\epsilon}_t\geq\frac{\epsilon}{2}\}\nonumber\\
    (B_\epsilon) & \text{  coincides with (\ref{eds}) on } &\{t, \tilde\lambda^{1,\epsilon}_t\leq \epsilon \text{ and } \tilde\lambda^{2,\epsilon}_t-\tilde\lambda^1_t\geq\epsilon\} .\nonumber
  \end{eqnarray}

Lemmas \ref{Aeta} and \ref{Beta} give   existence of global pathwise unique strong solutions to $(\hat A_\epsilon)$ and $(B_\epsilon)$ with any  random initial condition with ordered non-negative coordinates and independent from the driving Brownian motion.

  For $\xi\in\mathbb{R}_+^n$ deterministic with ordered coordinates, let $\hat\Lambda^{\epsilon,T,\xi}$ denotes the process solution to $(\hat A_\epsilon)$ on $[T,+\infty)$ starting from  $\xi$  at time $T$   and equal to $0$ on $(-\infty,T)$. Likewise,  $\tilde\Lambda^{\epsilon,T,\xi}$ denotes the process solution to $(B_\epsilon)$ on $[T,+\infty)$ starting from $\xi$ at time  $T$  and equal to $0$ on $(-\infty,T)$.  We distinguish   two cases to define by induction a solution to (\ref{eds}):
 
  \begin{enumerate}
      \item  on $\boldsymbol{\{\lambda^1_0\geq\epsilon\}}$ :
       we define by induction
      \begin{eqnarray}
      \tau^{\epsilon}_0 &= &0;\nonumber\\
      \tau^{\epsilon}_1 &= & \inf\left\{t\geq 0 : \hat\lambda^{1,\epsilon,0, \Lambda_0}_t\leq\frac{\epsilon}{2}\right\};\nonumber \\
      \mathcal{X}^{(1)}_t & = &\hat\Lambda^{\epsilon,0,\Lambda_0}_t\mathds{1}_{\{0\leq t\leq\tau^\epsilon_1\}} \text{ for all }t\in\mathbb{R};\nonumber\\
      \tau^\epsilon_2 & = & \inf\left\{t\geq\tau^{\epsilon}_1 : \tilde\lambda^{1,\epsilon,\tau_1^\epsilon, \mathcal{X}^{(1)}_{\tau_1^\epsilon}}_{t}\geq\epsilon\right\};\nonumber\\
      \mathcal{X}^{(2)}_t & = &\tilde\Lambda^{\epsilon,\tau_1^\epsilon,\mathcal{X}^{(1)}_{\tau_1^\epsilon}}_t\mathds{1}_{\{\tau_1^\epsilon<t\leq\tau^\epsilon_2\}}\text{ for all }t\in\mathbb{R};\nonumber\\
       & \vdots & \nonumber\\
      \tau^\epsilon_{2i+1} & = & \inf\left\{t\geq\tau^{\epsilon}_{2i} : \hat\lambda^{1,\epsilon,\tau_{2i}^\epsilon, \mathcal{X}^{(2i)}_{\tau_{2i}^\epsilon}}_{t}\leq\frac{\epsilon}{2}\right\};\nonumber\\
      \mathcal{X}^{(2i+1)}_t & = &\hat\Lambda^{\epsilon,\tau_{2i}^\epsilon,\mathcal{X}^{(2i)}_{\tau_{2i}^\epsilon}}_t\mathds{1}_{\{\tau_{2i}^\epsilon<t\leq\tau^\epsilon_{2i+1}\}}\text{ for all }t\in\mathbb{R};\nonumber\\
      \tau^{\epsilon}_{2i+2} &= &\inf\left\{t\geq \tau^{\epsilon}_{2i+1} : \tilde\lambda^{1,\epsilon,\tau_{2i+1}^\epsilon, \mathcal{X}^{(2i+1)}_{\tau_{2i+1}^\epsilon}}_t\geq\epsilon \right\};\nonumber \\
      \mathcal{X}^{(2i+2)}_t & = &\tilde\Lambda^{\epsilon,\tau_{2i+1}^\epsilon,\mathcal{X}^{(2i+1)}_{\tau_{2i+1}^\epsilon}}_t\mathds{1}_{\{\tau_{2i+1}^\epsilon<t\leq\tau^\epsilon_{2i+2}\}} \text{ for all }t\in\mathbb{R};\nonumber\\
       & \vdots & \nonumber 
      \end{eqnarray}
      and as for all $i \in\mathbb{N}$, the $\tau_i^\epsilon$ defined before are stopping times for the filtration $(\mathcal{F}_t)_{t\ge0}$, the random vectors $\mathds{1}_{\{\tau^{\epsilon}_i<+\infty\}}\mathcal{X}_{\tau_i^\epsilon}$ are $\mathcal{F}_{\tau_i^\epsilon}$-measurable, the construction makes sense.
      \item on $\boldsymbol{\{\lambda^1_0<\epsilon\}}$ :
      we define by induction
      \begin{eqnarray}
      \sigma^{\epsilon}_0 &= &0;\nonumber\\
      \sigma^{\epsilon}_1 &= & \inf\left\{t\geq 0 : \tilde\lambda^{1,\epsilon,0, \Lambda_0}_t\geq\epsilon\right\};\nonumber \\
      \mathcal{Y}^{(1)}_t & = &\tilde\Lambda^{\epsilon,0,\Lambda_0}_t\mathds{1}_{\{0\leq t\leq\sigma^\epsilon_1\}} \text{ for all }t\in\mathbb{R};\nonumber\\
      \sigma^\epsilon_2 & = & \inf\left\{t\geq \sigma^{\epsilon}_1 : \hat\lambda^{1,\epsilon,\sigma_{1}^\epsilon, \mathcal{Y}^{(1)}_{\sigma_{1}^\epsilon}}_t\leq\frac{\epsilon}{2} \right\};\nonumber\\
      \mathcal{Y}^{(2)}_t & = &\hat\Lambda^{\epsilon,\sigma_{1}^\epsilon,\mathcal{Y}^{(1)}_{\sigma_1^\epsilon}}_t\mathds{1}_{\{\sigma_{1}^\epsilon<t\leq\sigma^\epsilon_2\}}\text{ for all }t\in\mathbb{R};\nonumber\\
       & \vdots & \nonumber\\
      \sigma^\epsilon_{2i+1} & = & \inf\left\{t\geq  \sigma^{\epsilon}_{2i} : \tilde\lambda^{1,\epsilon,\sigma_{2i}^\epsilon, \mathcal{Y}^{(2i)}_{\sigma_{2i}^\epsilon}}_t\geq\epsilon \right\};\nonumber\\
      \mathcal{Y}^{(2i+1)}_t & = &\tilde\Lambda^{\epsilon,\sigma_{2i}^\epsilon,\mathcal{Y}^{(2i)}_{\sigma_{2i}^\epsilon}}_t\mathds{1}_{\{\sigma_{2i}^\epsilon<t\leq\sigma^\epsilon_{2i+1}\}}\text{ for all }t\in\mathbb{R};\nonumber\\
      \sigma^{\epsilon}_{2i+2} &= &\inf\left\{t\geq \sigma^{\epsilon}_{2i+1} : \hat\lambda^{1,\epsilon,\sigma_{2i+1}^\epsilon, \mathcal{Y}^{(2i+1)}_{\sigma_{2i+1}^\epsilon}}_{t}\leq\frac{\epsilon}{2}\right\};\nonumber \\
      \mathcal{Y}^{(2i+2)}_t & = &\hat\Lambda^{\epsilon,\sigma_{2i+1}^\epsilon,\mathcal{Y}^{(2i+1)}_{\sigma_{2i+1}^\epsilon}}_t\mathds{1}_{\{\sigma_{2i+1}^\epsilon<t\leq\sigma^\epsilon_{2i+2}\}} \text{ for all }t\in\mathbb{R};\nonumber\\
       & \vdots & \nonumber 
      \end{eqnarray}
       and as for all $i \in\mathbb{N}$, the $\sigma^\epsilon_i$ defined before are stopping times the filtration $(\mathcal{F}_t)_{t\ge0}$, the random vectors  $\mathds{1}_{\{\sigma^{\epsilon}_i<+\infty\}}\mathcal{Y}_{\sigma_i^\epsilon}$ are $\mathcal{F}_{\sigma_i^\epsilon}$-measurable, the construction makes sense.
  \end{enumerate}
  We finally define for all $\epsilon>0$ and $t\geq0$   :
   \begin{eqnarray}
      \label{presol}
      \mathcal{Z}^\epsilon_t &=& \left(\sum_{i=1}^{+\infty}\mathcal{X}^{(i)}_{t}\right)\mathds{1}_{\{\lambda^{\epsilon,1}_0\geq\epsilon\}}+\left(\sum_{i=1}^{+\infty}\mathcal{Y}^{(i)}_{t}\right)\mathds{1}_{\{\lambda^{\epsilon,1}_0<\epsilon\}}.\nonumber
  \end{eqnarray} 

   When $\lambda^1_0\geq\epsilon$ , for all $i\in\mathbb{N}$ and for $t\in[\tau^\epsilon_{2i+1},\tau^\epsilon_{2i+2})$, the equation for the smallest coordinate in $(B_\epsilon)$ and the non-negativity of $\tilde\lambda^{1,\epsilon,\tau_{2i+1}^\epsilon, \mathcal{X}^{(2i+1)}_{\tau_{2i+1}^\epsilon}}_t$ gives  
 \begin{eqnarray}
   d\tilde\lambda^{1,\epsilon,\tau_{2i+1}^\epsilon, \mathcal{X}^{(2i+1)}_{\tau_{2i+1}^\epsilon}}_t
   &\leq & 2\sqrt{\tilde\lambda^{1,\epsilon,\tau_{2i+1}^\epsilon, \mathcal{X}^{(2i+1)}_{\tau_{2i+1}^\epsilon}}_t}dB^1_t+(\alpha-(n-1)\beta)dt.\label{montees}
 \end{eqnarray}

 Then, according to the pathwise comparison theorem of Ikeda and Watanabe (that we recall in Theorem \ref{Ikeda} below), for all $i\in\mathbb{N}$ and for all $t\in[\tau^\epsilon_{2i+1},\tau^\epsilon_{2i+2})$ 
 \begin{eqnarray*}
     \tilde\lambda^{1,\epsilon,\tau_{2i+1}^\epsilon, \mathcal{X}^{(2i+1)}_{\tau_{2i+1}^\epsilon}}_t\leq r^{2i+2}_{t-\tau^{\epsilon}_{2i+1}}
 \end{eqnarray*}
 where for all $t\geq 0$
 \begin{equation*}
     r^{2i+2}_t = \frac{\epsilon}{2}+2\int_{\tau^\epsilon_{2i+1}}^{\tau^\epsilon_{2i+1}+t}\sqrt{r_s}dB^1_s+(\alpha-(n-1)\beta)t,
 \end{equation*}
 which is a CIR process. When $\lambda^1_0<\epsilon$, the same kind of comparison can be made on $[\sigma^\epsilon_{2i},\sigma^\epsilon_{2i+1})$ with a CIR process that we will call $r^{2i+1}$ for all $i\in\mathbb{N}^*$.
 
 For $\lambda^1_0\ge\epsilon$, the times $\tau^\epsilon_{2i+2}-\tau^\epsilon_{2i+1}$ for all $i\in\mathbb{N}$  are larger than the time interval for the  CIR processes  $r^{2i+2}$  to go from $\frac{\epsilon}{2}$ to $\epsilon$. Moreover, the times for the $r^{2i+2}$ processes to go from $\frac{\epsilon}{2}$ to $\epsilon$ are iid positive random variables. Consequently, there is no accumulation of the stopping times  $\tau^j_\epsilon$ which go to infinity as $j\rightarrow\infty$. The same argument applies when $\lambda^1_0<\epsilon$ with the $\sigma^j_\epsilon$ which also go to infinity as $j\rightarrow\infty$.
   
   The stochastic process $\mathcal{Z}^\epsilon$ is thus defined globally in time.

   We recall that ($\hat A_\epsilon$) and ($B_\epsilon$) respectively coincide with ($\ref{eds}$) when $\hat\lambda^{1,\epsilon}_t\geq\frac{\epsilon}{2}$ and when $\tilde\lambda^{1,\epsilon}_t\leq\epsilon$ and $\tilde\lambda^{2,\epsilon}_t-\tilde\lambda^{1,\epsilon}_t\geq\epsilon$. On the other hand, when $\lambda^1_0\geq\epsilon$, on $[\tau^\epsilon_{2i},\tau^\epsilon_{2i+1}]$, $\mathcal{Z}^\epsilon$ evolves according to ($\hat A_\epsilon$) and $\mathcal{Z}^{1,\epsilon}\geq\frac{\epsilon}{2}$ while on $[\tau^\epsilon_{2i+1},\tau^\epsilon_{2i+2}]$, $\mathcal{Z}^\epsilon$ evolves according to ($B_\epsilon$) and $\mathcal{Z}^{1,\epsilon}\leq\epsilon$.  By induction on $i$ we deduce that $\mathcal{Z}^\epsilon$ is a solution to (\ref{eds}) until
   
   \[\inf\left\{t\in\bigcup_{i\in\mathbb{N}}[\tau^\epsilon_{2i+1},\tau^\epsilon_{2i+2}], \mathcal{Z}^{2,\epsilon}_t-\mathcal{Z}^{1,\epsilon}_t\leq\epsilon\right\}\geq \inf\{t\geq0 :  \mathcal{Z}^{1,\epsilon}_t\leq\epsilon \text{ and }\mathcal{Z}^{2,\epsilon}_t-\mathcal{Z}^{1,\epsilon}_t\leq\epsilon\} =:\zeta_\epsilon.\]  Likewise, when $\lambda^1_0<\epsilon$, $\mathcal{Z}^\epsilon$ is a solution to (\ref{eds}) until
   \[\inf\left\{t\in\bigcup_{i\in\mathbb{N}}[\sigma^\epsilon_{2i},\sigma^\epsilon_{2i+1}], \mathcal{Z}^{2,\epsilon}_t-\mathcal{Z}^{1,\epsilon}_t\leq\epsilon\right\}\geq\inf\{t\geq0 :  \mathcal{Z}^{1,\epsilon}_t\leq\epsilon \text{ and }\mathcal{Z}^{2,\epsilon}_t-\mathcal{Z}^{1,\epsilon}_t\leq\epsilon\} =: \zeta_\epsilon.\] 
   
  From Lemmas \ref{Aeta} and \ref{Beta}, we have when $\lambda^1_0\geq\epsilon$ :
   \begin{equation}
        \mathbb{P}\{\exists i\in\mathbb{N},\exists t\in(\tau_{2i}^\epsilon, \tau_{2i+1}^\epsilon]:\mathcal{Z}^{i,\epsilon}_t=\mathcal{Z}^{i+1,\epsilon}_t \text{ and } \mathcal{Z}^{j,\epsilon}_t=\mathcal{Z}^{j+1,\epsilon}_t \text{ for } 0\leq i<j\leq n-1\}=0\label{iiiii}
    \end{equation}
    and
   \begin{eqnarray}
     \label{iiiiii}
  \mathbb{P}\{\exists i\in\mathbb{N},\exists t\in(\tau_{2i+1}^\epsilon, \tau_{2i+2}^\epsilon]:\mathcal{Z}^{i,\epsilon}_t=\mathcal{Z}^{i+1,\epsilon}_t \text{ and } \mathcal{Z}^{j,\epsilon}_t=\mathcal{Z}^{j+1,\epsilon}_t \text{ for } 2\leq i<j\leq n-1\}=0
\end{eqnarray}
where by convention $\mathcal{Z}^{0,\epsilon}\equiv0$.

On the time intervals $[\tau_{2i+1}^\epsilon\wedge\zeta_\epsilon, \tau_{2i+2}^\epsilon\wedge\zeta_\epsilon]$, we have $ \mathcal{Z}^{2,\epsilon}_t-\mathcal{Z}^{1,\epsilon}_t\geq\epsilon$. This together with (\ref{iiiii}-\ref{iiiiii}) allows to conclude the proof of $(iii)$ when $\lambda^1_0\geq\epsilon$. The same reasoning can be made when $\lambda^1_0<\epsilon$.
Consequently,
\begin{equation}
\label{Zmultiplecol}
    \mathbb{P}\Big\{\exists t\in (0,\zeta_\epsilon]: \mathcal{Z}^{i,\epsilon}_t = \mathcal{Z}^{i+1,\epsilon}_t \text{ and } \mathcal{Z}^{j,\epsilon}_t = \mathcal{Z}^{j+1,\epsilon}_t \text{ for some }1\leq i<j\leq n-1\Big\}=0
\end{equation}

   As the solutions to equation (\ref{eds}) are pathwise unique (see Lemma \ref{Beta}), for $n\in\mathbb{N}^*$, the processes $\mathcal{Z}^{\frac{1}{n}}$ and $\mathcal{Z}^{\frac{1}{n+1}}$ coincide on $\left[0,\zeta_{\frac{1}{n}}\wedge\zeta_{\frac{1}{n+1}}\right]$. Thus, $\zeta_{\frac{1}{n}}\wedge\zeta_{\frac{1}{n+1}}=\zeta_{\frac{1}{n}}$ and the sequence $(\zeta_{\frac{1}{n}})_{n\in\mathbb{N}^*}$  is non-decreasing. Moreover, for all $n\in\mathbb{N}^*$, $\mathcal{Z}^{\frac{1}{n}}$  verifies (\ref{Zmultiplecol}). Consequently,  we can define for all $t\in[0,\underset{\epsilon\rightarrow0}{\lim}\zeta_\epsilon)$
  \begin{equation}
  \label{solution}
      \Lambda_t =\mathcal{Z}^1_{t}\mathds{1}_{\left\{0\leq t\leq\zeta_{1}\right\}}+ \sum_{n\geq1}\mathcal{Z}^{\frac{1}{n+1}}_t\mathds{1}_{\left\{\zeta_{\frac{1}{n}}<t\leq\zeta_{\frac{1}{n+1}}\right\}}
  \end{equation}
  which is a solution to SDE (\ref{eds}) on $[0,\underset{\epsilon\rightarrow0}{\lim}\zeta_\epsilon)$ verifying $(iii)$ of Theorem \ref{existence}.

 Finally,
   as the solutions to (\ref{eds}) are pathwise unique   (Lemma \ref{pathwise_uniqueness}), we can apply the Yamada-Watanabe theorem (see for instance \cite[Theorem 1.7 p368]{MR1725357}) to deduce the existence of strong solutions to the equation.
  Lemma \ref{4} shows that the condition 
  \begin{equation*}
    \text{ for all }t\geq0, \sum_{i=1}^{n-1}\int_0^t\frac{\lambda^{i+1}_s}{\lambda^{i+1}_s-\lambda^{i}_s}ds<\infty \text{ a.s..}
\end{equation*}
  is satisfied.

\paragraph{}
  
   Since on $\{\underset{\epsilon\rightarrow0}{\lim}\zeta_\epsilon<+\infty\}$ we have $\lambda^1_{\zeta_\epsilon}+\lambda^2_{\zeta_\epsilon}= 2\lambda^1_{\zeta_\epsilon}+\lambda^2_{\zeta_\epsilon}-\lambda^1_{\zeta_\epsilon}\leq3\epsilon$, $\underset{t\in[0,\underset{\epsilon\rightarrow0}{\lim}\zeta_\epsilon)}{\inf}\lambda^1_t+\lambda^2_t=0$,
   we then use Proposition \ref{multiplecollision} $(ii)$ with $k=2$ to conclude that $\underset{\epsilon\rightarrow0}{\lim}\zeta_\epsilon=+\infty$ when $\alpha-(n-1)\beta\geq1-\beta$ which is $(i)$ from Theorem \ref{existence}.

   For  $\alpha-(n-1)\beta<1-\beta$ and $\gamma\geq0$, let us prove asumption $(ii)$ of Theorem \ref{existence}. 
   Following the steps of the proof of Proposition \ref{multiplecollision} with $k=2$ until (\ref{aaaaa}), we have for all $0\leq t < \underset{\epsilon\rightarrow0}{\lim}\zeta_\epsilon$ :
   \begin{eqnarray}
  d(\lambda^{1}_t+\lambda^2_t)  &\leq& 2\sqrt{\lambda^{1}_t+\lambda^2_t}dW^2_t-2\gamma(\lambda^{1}_t+\lambda^2_t)dt+2(\alpha-(n-2)\beta)dt\nonumber
  \end{eqnarray}
  where by Lévy's characterization, $W^2$ defined by
  $W^2_0=0$ and  \\$dW^2_t=\mathds{1}_{\left\{0\leq t<\underset{\epsilon\rightarrow0}{\lim}\zeta_\epsilon\right\}}\frac{\sqrt{\lambda^1_t}dB^1_t+\sqrt{\lambda^2_t}dB^2_t}{\sqrt{\lambda^1_t+\lambda^2_t}}+\mathds{1}_{\left\{t\geq\underset{\epsilon\rightarrow0}{\lim}\zeta_\epsilon\right\}}\frac{dB^1_t+dB^2_t}{\sqrt{2}}$ is a Brownian motion.

By the pathwise comparison theorem of Ikeda and Watanabe (that we recall in Theorem \ref{Ikeda} below),
$$\lambda^1_t+\lambda^2_t\leq r_t \text{ for all } 0\leq t < \underset{\epsilon\rightarrow0}{\lim}\zeta_\epsilon \text{ a.s.}$$ where  for all $t\geq0$
\begin{equation}
\label{r}
    r_t= \lambda^1_0+\lambda^2_0+2\int_0^t\sqrt{r_s}dW^2_s-2\gamma \int_0^tr_sds+2(\alpha-(n-2)\beta)t
\end{equation}

 is a CIR process. Applying  Lemma \ref{lamberton} with $a=2(\alpha-(n-2)\beta)$, $b=2\gamma$ and $\sigma=2$ which satisfy $a<\frac{\sigma^2}{2}$ and $b\geq0$, we  conclude that the hitting time of $0$ by $r$ is finite almost surely. Consequently,   $\mathbb{P}\left(\underset{\epsilon\rightarrow0}{\lim}\zeta_\epsilon=+\infty\right)=0$ which concludes the proof of Theorem \ref{existence} $(ii)$.

  \begin{lemma}
  \label{Aeta}
  Let us assume $\alpha-(n-1)\beta>0$.
  The SDE
   \begin{equation}
  \label{hatAeta}
         d \hat\lambda^{i,\epsilon}_t = 2\sqrt{\hat\lambda^{i,\epsilon}_t}dB^i_t+\left[\alpha-(n-1)\beta+1-0\vee\frac{2\sqrt{2}}{\sqrt{\epsilon}}\left(\sqrt{\hat\lambda^{i,\epsilon}_t}-\frac{\sqrt{\epsilon}}{2\sqrt{2}}\right)\wedge1 -2\gamma\hat\lambda^{i,\epsilon}_t+2\beta\hat\lambda^{i,\epsilon}_t\sum_{j\neq i}\frac{1}{\hat\lambda^{i,\epsilon}_t-\hat\lambda^{j,\epsilon}_t}\right]dt\tag{$\hat{A} _\epsilon$}
  \end{equation}
  \[0\leq \hat\lambda^{1,\epsilon}_t<\dots< \hat\lambda^{n,\epsilon}_t, \text{ a.s., } dt-a.e.\]
  
    has a global pathwise unique strong solution $(\hat{\lambda}^{1,\epsilon}_t,\dots,\hat{\lambda}^{1,\epsilon}_t)_{t\geq0}$ starting from any random  initial condition $\Lambda_0=(\lambda^1_0,\dots,\lambda^n_0)$ independent from $\mathbf{B}$ such that $0\leq \lambda^1_0\leq\dots\leq\lambda^n_0$ a.s..
    
    Moreover,
    \begin{equation}
    \label{AetaNonMultipleCollision}
        \mathbb{P}\{\exists t>0:\hat\lambda^{i,\epsilon}_t=\hat\lambda^{i+1,\epsilon}_t \text{ and } \hat\lambda^{j,\epsilon}_t=\hat\lambda^{j+1,\epsilon}_t \text{ for } 0\leq i<j\leq n-1\}=0
    \end{equation}
    where by convention $\hat\lambda^{0,\epsilon}\equiv0$.
  \end{lemma}
  
  \begin{proof}
  Let us consider  $\mathcal{F}_t=\sigma\left((\sqrt{\lambda^1_0},\dots,\sqrt{\lambda^n_0}), \left(\mathbf{B}_s \right)_{s \le t} \right)$    and the SDE defined by 
  \begin{equation}
  \label{Aetaeq}
     dx^{i,\epsilon}_t  =  dB^i_t +\frac{\alpha-(n-1)\beta-0\vee\frac{2\sqrt{2}}{\sqrt{\epsilon}}\left(x^{i,\epsilon}_t-\frac{\sqrt{\epsilon}}{2\sqrt{2}}\right)\wedge 1}{2x^{i,\epsilon}_t}dt-\gamma x^{i,\epsilon}_tdt+\beta x^{i,\epsilon}_t\sum_{j\neq i}\frac{dt}{(x^{i,\epsilon}_t)^2-(x^{j,\epsilon}_t)^2}\text{ for }i\in\{1,\dots,n\}\tag{$A_\epsilon$}
    \end{equation}
    \[0\leq x^{1,\epsilon}_t<\dots< x^{n,\epsilon}_t, \text{ }a.s., dt-a.e.\nonumber\]
    with  random initial condition $(\sqrt{\lambda^1_0},\dots,\sqrt{\lambda^n_0})$ such that $0\leq \sqrt{\lambda^1_0}\leq\dots\leq \sqrt{\lambda^n_0}$.

  We are going to apply  Cepa's multivoque equations theory (\cite{MR1459451}) to conclude that there exists a unique strong solution to (\ref{Aetaeq}). 
  
  To do so, we define
  \begin{eqnarray}
  D&=&\{0<x^1<x^2<\dots<x^n\}\nonumber\\
  \Phi_\gamma  &:&  (x^1,\dots,x^n)\in \mathbb{R}^n \rightarrow 
\left\{
\begin{array}{c }\displaystyle
    -\sum_{i=1}^n\left[\frac{\alpha-(n-1)\beta}{2}\ln |x^i|-\frac{\gamma}{2} (x^i)^2+\frac{\beta}{4}\sum_{j\ne i}\left(\ln| x^{i}-x^{j}| + \ln| x^{i}+x^{j}|\right) \right]\text{ if } x\in D \\
    +\infty \text{ if } x\notin D 
\end{array}
\right.\nonumber\\
g &:&  (x^1,\dots,x^n)\in \mathbb{R}^n \rightarrow 
    \left(-\frac{0\vee\frac{2\sqrt{2}}{\sqrt{\epsilon}}\left(x^{1,\epsilon}_t-\frac{\sqrt{\epsilon}}{2\sqrt{2}}\right)\wedge 1}{2x^{1,\epsilon}_t},\dots,-\frac{0\vee\frac{2\sqrt{2}}{\sqrt{\epsilon}}\left(x^{n,\epsilon}_t-\frac{\sqrt{\epsilon}}{2\sqrt{2}}\right)\wedge 1}{2x^{n,\epsilon}_t}\right) \nonumber
  \end{eqnarray}
  to rewrite the system of SDE on D with $X=(x^{1,\epsilon},\dots,x^{n,\epsilon})$  the following way
  
  \begin{equation}
 dX_t  =  d\mathbf{B}_t+g(X_t)dt-\nabla\Phi_\gamma(X_t)dt.\tag{$A_\epsilon$}
  \end{equation}
  Since $g$ is globally Lipschitz and $\Phi_\gamma$ is convex, Cépa's multivoque equations theory shows existence and uniqueness of a strong solution to equation
 
  \begin{align}
   d\tilde X_t & =  d\mathbf{B}_t+g(\tilde X_t)dt-\nabla\Phi_\gamma(\tilde X_t)dt-\nu(\tilde X_t)dL_t \text{ for all  }t\geq0\label{tildeAeta}\tag{$\tilde A_\epsilon$}\\
   & \forall t\geq0,\tilde X_t\in \Bar D \text{ a.s.}\nonumber\\
   &\tilde X_0=(\sqrt{\lambda^1_0},\dots,\sqrt{\lambda^n_0})\nonumber
  \end{align}
  where $\tilde X$  is a continuous adapted to $(\mathcal{F}_t)_{t\geq0}$ process,   $L$ is a continuous non-decreasing adapted to $(\mathcal{F}_t)_{t\geq0}$ process with $L_0=0$  verifying
  \begin{equation}
      L_t= \int_0^t\mathds{1}_{\{\tilde X_s\in\partial D\}}dL_s,\nonumber
  \end{equation}
  and
   $\nu(x)\in\pi(x)$ ($\pi(x)$ is the set of unitary outward normals to $\partial D$ at $x\in\partial D$). The solution to equation (\ref{tildeAeta}) follows the conditions : for all $t>0$ 
  \begin{eqnarray}
  \mathbb{E}\left[\int_0^t\mathds{1}_{\{\tilde X_s\in\partial D\}}ds\right] & = & 0,\nonumber\\
  \mathbb{E}\left[\int_0^t|\nabla\Phi_\gamma(\tilde X_s)|ds\right] & < & \infty.\nonumber
  \end{eqnarray}

  We apply \cite[Theorem 2.2]{MR1875671} which is an application of Cépa's multivoque equations theory to this kind of SDE and the remark following \cite[Theorem 3.1]{MR1459451} to deduce that $(\tilde A_\epsilon)$  has a unique strong solution. Let us now prove that the boundary process $L$ is equal to zero.

 For all $m\in\{1,\dots,n\}$, for all $t\geq0$, we have  with $C = \frac{\alpha-(n-1)\beta}{2}$  :
 \begin{equation}
    \label{H}
     \int_0^t\left|\frac{C}{x^{m,\epsilon}_s}+\beta x^{m,\epsilon}_s\sum_{j\neq m}\frac{1}{(x^{m,\epsilon}_s)^2-(x^{j,\epsilon}_s)^2}\right|ds<\infty
 \end{equation}
Let us prove by  backward induction on $m$ that 
\begin{equation}
    \label{Hrec}
    \text{for all $ 1 <m\leq n$ and for all $t\geq0$, }
    \int_0^t\left|\frac{1}{x^{m,\epsilon}_s}\right|+\sum_{l<m}\left|\frac{x^{m,\epsilon}_s}{(x^{m,\epsilon}_s)^2-(x^{l,\epsilon}_s)^2}\right|ds  <\infty \tag{$H_m$}
\end{equation}
\begin{itemize}
    \item $m=n$
\end{itemize}
As all the terms in the absolute value of (\ref{H})  have the same sign   for $m=n$,  we  deduce the individual integrability.

\begin{itemize}
    \item Let $1<m\leq n$ and let us assume ($H_j$) for all $j\in\{m+1,\dots,n\}$.
\end{itemize}
We have :
\begin{eqnarray}
  \frac{C}{x^{m,\epsilon}_s}+\beta x^{m,\epsilon}_s\sum_{j\neq m}\frac{1}{(x^{m,\epsilon}_s)^2-(x^{j,\epsilon}_s)^2}&=&\frac{C}{x^{m,\epsilon}_s}+\beta x^{m,\epsilon}_s\sum_{j< m}\frac{1}{(x^{m,\epsilon}_s)^2-(x^{j,\epsilon}_s)^2}-\beta \sum_{j>m}\frac{x^{m,\epsilon}_s}{(x^{j,\epsilon}_s)^2-(x^{m,\epsilon}_s)^2}\nonumber
\end{eqnarray}

Let us remark that for all $s\geq0$
\begin{eqnarray}
\label{aa}
  \frac{C}{x^{m,\epsilon}_s}+\beta x^{m,\epsilon}_s\sum_{j<m}\frac{1}{(x^{m,\epsilon}_s)^2-(x^{j,\epsilon}_s)^2}\leq\left|\frac{C}{x^{m,\epsilon}_s}+\beta x^{m,\epsilon}_s\sum_{j\neq m}\frac{1}{(x^{m,\epsilon}_s)^2-(x^{j,\epsilon}_s)^2}\right|+\beta \sum_{j>m}\frac{x^{j,\epsilon}_s}{(x^{j,\epsilon}_s)^2-(x^{m,\epsilon}_s)^2}
\end{eqnarray}
by the triangle inequality and since $x^{j,\epsilon}_s\geq x^{m,\epsilon}_s$ for $j>m$.
 
By (\ref{H}) and the induction hypothesis for $j\in\{m+1,\dots,n\}$, each term in the right-hand side  of (\ref{aa}) is integrable, which ends the induction argument.

 Consequently,  for all $1\leq l <m\leq n$  and for all $t\geq0$  we have
  
  \begin{eqnarray}
  \label{argoccup}
  \int_0^t\frac{1}{x^{m,\epsilon}_s-x^{l,\epsilon}_s}ds=\int_0^t\frac{x^{m,\epsilon}_s+x^{l,\epsilon}_s}{(x^{m,\epsilon}_s)^2-(x^{l,\epsilon}_s)^2}ds\leq 2 \int_0^t\frac{x^{m,\epsilon}_s}{(x^{m,\epsilon}_s)^2-(x^{l,\epsilon}_s)^2}ds <\infty. 
  \end{eqnarray}
  
  As in the second part of the proof of \cite[Theorem 2.2]{MR1875671} (equation (2.40)), using the occupation times formula and (\ref{argoccup}), we have for $1\leq l<m\leq n,t\geq 0$
  
  \begin{equation}
      \int_0^{+\infty}\frac{L^a_t(x^{m,\epsilon}-x^{l,\epsilon})}{a}da =\int_0^t\frac{d\langle x^{m,\epsilon}-x^{l,\epsilon}\rangle_s}{x^{m,\epsilon}_s-x^{l,\epsilon}_s}= 2\int_0^t\frac{1}{x^{m,\epsilon}_s-x^{l,\epsilon}_s}ds<+\infty
  \end{equation} 
  and \[\int_0^{+\infty}\frac{L^a_t(x^{1,\epsilon})}{a}da=\int_0^t\frac{d\langle x^{1,\epsilon}\rangle_s}{x^{1,\epsilon}_s} = \int_0^t\frac{1}{x^{1,\epsilon}_s}ds<+\infty\]
  where $L_t^a(\mathcal{X})$ is the   local time at time $t$ and on level $a$ for a real continuous semimartingale $\mathcal{X}$. Since the function $a\mapsto\frac{1}{a}$ is not integrable at 0 and $(L_t^a(X))$ is cadlag in $a$ by \cite[Theorem 1.7 p225]{MR1725357}, one deduces that $L^0_t(x^{m,\epsilon}-x^{l,\epsilon})=L^0_t(x^{1,\epsilon})=0$. 
  
  From there, the reasoning made in the proof of \cite[Theorem 2]{MR1875671} allows to conclude that the boundary process $L$ is equal to zero.

  Then, with $\hat\lambda^{i,\epsilon}=(x^{i,\epsilon})^2$ for all $i\in\{1,\dots,n\}$ we obtain a global solution to (\ref{hatAeta}). Following the same approach used to demonstrate Lemma \ref{Beta} below, the solutions to  (\ref{hatAeta}) are pathwise unique. The Yamada-Watanabe Theorem (see for instance \cite[Theorem 1.7 p368]{MR1725357}), allows to conclude that (\ref{hatAeta}) has a pathwise unique global strong solution.
  
  Let us now prove $(\ref{AetaNonMultipleCollision})$.
  
  Let us consider the SDE defined by ($A_\epsilon$)  with initial condition $0\leq \sqrt{\lambda^1_0}\leq\dots\leq \sqrt{\lambda^n_0}$.

    Let us define for all $\epsilon>0,M>0$  \[\tau_M=\inf\{t\geq0: \exists i\in\{1,\dots,n\},  x^{i,\epsilon}_t\geq M\},\] and and for  $t\in[0;\tau_M)$ :   $\Theta(t)=(\theta_1(t),\dots,\theta_n(t))$ with
  \begin{eqnarray}
    \forall i\in\{1,\dots, n \}\text{, } \theta_i(t) & = &-\frac{0\vee\frac{2\sqrt{2}}{\sqrt{\epsilon}}\left( x^{i,\epsilon}_t-\frac{\sqrt{\epsilon}}{2\sqrt{2}}\right)\wedge 1}{2 x^{i,\epsilon}_t} +\gamma  x^{i,\epsilon}_t\nonumber
  \end{eqnarray}
  and for all $t\geq0$
  \begin{equation}
      Z(t)=\exp\left\{-\int_0^{t\wedge\tau_M}\Theta(u)\cdot d{\textbf{B}}_u-\frac{1}{2}\int_0^{t\wedge\tau_M}||\Theta(u)||^2du\right\}. \nonumber
  \end{equation}

  We have for all $i\in\{1,\dots,n\}$,
  \begin{eqnarray}
    \theta_i^2(t) & = &\left(-\frac{0\vee\frac{2\sqrt{2}}{\sqrt{\epsilon}}\left( x^{i,\epsilon}_t-\frac{\sqrt{\epsilon}}{2\sqrt{2}}\right)\wedge 1}{2 x^{i,\epsilon}_t} +\gamma  x^{i,\epsilon}_t\right)^2\nonumber\\
    & \leq & 2\left(\frac{0\vee\frac{2\sqrt{2}}{\sqrt{\epsilon}}\left( x^{i,\epsilon}_t-\frac{\sqrt{\epsilon}}{2\sqrt{2}}\right)\wedge 1}{2 x^{i,\epsilon}_t}\right)^2+2(\gamma  x^{i,\epsilon}_t)^2\nonumber\\
    &\leq &\frac{1}{\epsilon}+2\gamma^2M^2.\nonumber
  \end{eqnarray}

  We thus have
  \begin{eqnarray}
    \mathbb{E}\left[\exp\left\{\frac{1}{2}\int_0^{t\wedge\tau_M}||\Theta(u)||^2du\right\}\right]& <&\infty \text{ for all }t\geq0.\nonumber
  \end{eqnarray} 
  Then, according to Novikov's criterion (see for instance \cite[Proposition 5.12 p198]{MR1121940}), $Z$ is a $\mathbb{P}$-martingale, and $\mathbb{E}[Z(t)]=1$. Consequently, recalling that recall that $\mathcal{F}_t=\sigma\left((\sqrt{\lambda^1_0},\dots,\sqrt{\lambda^n_0}), \left(\mathbf{B}_s \right)_{s \le t} \right)$  and defining $\mathbb{Q}$ such   that
  $$\frac{d\mathbb{Q}}{d\mathbb{P}}_{|\mathcal{F}_t}=Z(t)$$ and
  \begin{eqnarray}
      \text{ for all } i\in\{1,\dots,n\},\text{ } \check{ B}^{i,M}_t  &=& B^i_t+\int_0^{t\wedge \tau_M}\theta_i(s)ds \nonumber \\
      & = &  B^i_t-\int_0^{t\wedge \tau_M}\frac{0\vee\frac{2\sqrt{2}}{\sqrt{\epsilon}}\left( x^{i,\epsilon}_s-\frac{\sqrt{\epsilon}}{2\sqrt{2}}\right)\wedge 1}{2 x^{i,\epsilon}_s} +\gamma  x^{i,\epsilon}_sds, \text{ }t\geq 0,\nonumber
  \end{eqnarray}
  ${\Check{\textbf{B}}^M}=({\check{ B}^{1,M}},\dots,{\check{ B}^{n,M}})$ is a $\mathbb{Q}$- Brownian motion
  according to the Girsanov theorem (see for instance \cite[Proposition 5.4 p194]{MR1121940}).
  
  Thus,  ($A_\epsilon$) can be rewritten in terms of $\Check{\mathbf{B}}^M$ as 
   \begin{equation}
  \label{Aetaeqmod}
     d x^{i,\epsilon}_t  =  d\check B^{i,M}_t +\frac{\alpha-(n-1)\beta}{2 x^{i,\epsilon}_t}dt-\gamma x^{i,\epsilon}_t\mathds{1}_{\{t\geq\tau_M\}}dt   +\beta  x^{i,\epsilon}_t\sum_{j\neq i}\frac{dt}{( x^{i,\epsilon}_t)^2-( x^{j,\epsilon}_t)^2}\text{ for }i\in\{1,\dots,n\}.\tag{$\check{A}_{\epsilon,M}$}
    \end{equation}
     By the same arguments as in the beginning of the proof, the SDE

    \begin{eqnarray}
    \label{proxy}
       d X_t & =&  d\check{\mathbf{B}}^M_t-\nabla\Phi_0( X_t)dt-\nu( X_t)dL_t \text{ for all  }t\geq0\\
   &&  \forall t\geq0, \text{ } X_t\in \Bar D \text{ a.s.}\nonumber\\
   && X_0=(\sqrt{\lambda^1_0},\dots,\sqrt{\lambda^n_0})\nonumber
    \end{eqnarray}
    admits a global solution and the term $\nu( X_t)dL_t$ is zero. We can apply \cite[Theorem 3.1]{MR2792586} to the SDE (\ref{proxy}) to conclude that its solutions cannot have multiple collisions. This last SDE can  be rewritten
    \begin{eqnarray*}
     d x^{i,\epsilon}_t  &=&  d\check B^{i,M}_t +\frac{\alpha-(n-1)\beta}{2 x^{i,\epsilon}_t}dt+\beta  x^{i,\epsilon}_t\sum_{j\neq i}\frac{dt}{( x^{i,\epsilon}_t)^2-( x^{j,\epsilon}_t)^2}\text{ for }i\in\{1,\dots,n\}\\
     &&0\leq x^{1,\epsilon}<\dots<x^{n,\epsilon}, \text{ a.s., } dt-a.e..
    \end{eqnarray*}
  By  pathwise uniqueness (Lemma \ref{pathwise_uniqueness}), the solutions to this last SDE coincide with the solutions to (\ref{Aetaeqmod})  on $[0,\tau_M)$, which implies that there is no collision  of $(x^{1,\epsilon}_t,\dots,x^{n,\epsilon}_t)_{t\geq0}$ on $[0,\tau_M)$ under the probability $\mathbb{Q}$. There is thus no multiple collision of $(x^{1,\epsilon}_t,\dots,x^{n,\epsilon}_t)_{t\geq0}$ under the probability $\mathbb{P}$  on $[0,\tau_M)$ :
  \begin{equation*}
        \mathbb{P}\{\exists t\in[0,\tau_M):x^{i,\epsilon}_t=x^{i+1,\epsilon}_t \text{ and } x^{j,\epsilon}_t=x^{j+1,\epsilon}_t \text{ for } 0\leq i<j\leq n-1\}=0.
  \end{equation*}
  
  Since it is true for all $M>0$, and since as the SDE ($A_\epsilon$) admits a continuous global in time solution, $\tau_M\underset{M\rightarrow+\infty}{\longrightarrow}+\infty\text{ }\mathbb{P}-a.s.$, we have the result for ($A_\epsilon$), and thus for ($\hat A_\epsilon$), which concludes the proof.
  
  \end{proof}
 
  \begin{lemma}
  \label{Beta}
Let us assume $\alpha-(n-1)\beta>0$.  The SDE $(B_\epsilon)$ with random initial condition $(\tilde\lambda^{1,\epsilon}_0,\dots,\tilde\lambda^{n,\epsilon}_0)$ such that $0\leq\tilde\lambda^{1,\epsilon}_0\leq\dots\leq\tilde\lambda^{n,\epsilon}_0 $a.s. and independent from the Brownian motion $\mathbf{B}=(B^1,\dots,B^n)$ has a global pathwise unique strong solution $(\tilde{\lambda}^{1,\epsilon}_t,\dots,\tilde{\lambda}^{1,\epsilon}_t)_{t\geq0}$.

Moreover, 

\begin{eqnarray}
  \label{multiplecollisionbeta}
  \mathbb{P}\{\exists t>0:\tilde\lambda^{i,\epsilon}_t=\tilde\lambda^{i+1,\epsilon}_t \text{ and } \tilde\lambda^{j,\epsilon}_t=\tilde\lambda^{j+1,\epsilon}_t \text{ for } 2\leq i<j\leq n-1\}=0.
\end{eqnarray}
  \end{lemma}
  
  \begin{proof}

  Let us consider a  Brownian motion $\Tilde{\textbf{B}}=(\tilde B^1,\dots,\tilde B^n)$, and the SDE 
  
  \begin{eqnarray}
  \label{1}
     d\tilde\lambda^{1,\epsilon}_t & =& 2\sqrt{\tilde\lambda^{1,\epsilon}_t}d\tilde B^1_t+(\alpha-(n-1)\beta)dt-2\gamma\tilde \lambda^{1,\epsilon}_tdt  \\
   \forall i\in\{2,\dots, n \} \text{ , }  d \tilde\lambda^{i,\epsilon}_t &=& 2\sqrt{\tilde\lambda^{i,\epsilon}_t}d\tilde B^i_t+(\alpha-(n-1)\beta+1)dt-\left[0\vee\frac{2}{\sqrt{\epsilon}}\left(\sqrt{\tilde\lambda^{i,\epsilon}_t}-\frac{\sqrt{\epsilon}}{2}\right)\wedge1\right]dt\nonumber\\ &&-2\gamma\tilde \lambda^{i,\epsilon}_tdt+2\beta\sum_{j\geq 2, j\neq i}\frac{\tilde\lambda^{i,\epsilon}_t}{\tilde\lambda^{i,\epsilon}_t-\tilde\lambda^{j,\epsilon}_t}dt.\nonumber
  \end{eqnarray}
  Let us remark that $\tilde\lambda^1$ is a CIR process, and that the coordinates $i\in\{2,\dots,n\}$ satisfy an autonomous SDE for $n-1$ particles similar to equation $(\hat A_\epsilon)$. The only differences are  the coefficient $(n-1)\beta$ in $(\hat A_\epsilon)$ which remains $(n-1)\beta$ here and is thus unaffected by the change in the number of particles, and the terms $\left[0\vee\frac{2\sqrt{2}}{\sqrt{\epsilon}}\left(\sqrt{\tilde\lambda^{i,\epsilon}_t}-\frac{\sqrt{\epsilon}}{2\sqrt{2}}\right)\wedge1\right]$ in $(\hat A_\epsilon)$ which becomes $\left[0\vee\frac{2}{\sqrt{\epsilon}}\left(\sqrt{\tilde\lambda^{i,\epsilon}_t}-\frac{\sqrt{\epsilon}}{2}\right)\wedge1\right]$ here. We can still consider the root process $(\sqrt{\tilde\lambda^{2,\epsilon}_t},\dots,\sqrt{\tilde\lambda^{n,\epsilon}_t})$, and apply the same method as in the proof of Lemma \ref{Aeta} to prove the existence of a pathwise unique strong solution to this subsystem. Equation (\ref{1}) also has a strong solution (Lemma \ref{lamberton}), and consequently the whole n-particles system considered here admits a global strong solution.
  
  Let us define for all $\epsilon>0$ and for  $t\geq0$ :   $\Theta(t)=(\theta_1(t),\dots,\theta_n(t))$ with
  \begin{eqnarray}
    \theta_1(t) &=& -\frac{\beta}{\sqrt{\tilde\lambda^{1,\epsilon}_t}}\sum_{j\neq1}\frac{\tilde\lambda^{1,\epsilon}_t\wedge\epsilon}{(\tilde\lambda^{j,\epsilon}_t-\tilde\lambda^{1,\epsilon}_t\wedge\epsilon)\vee\epsilon}\nonumber \\
    \forall i\in\{2,\dots, n \}\text{, } \theta_i(t) & = &\frac{\beta}{\sqrt{\tilde\lambda^{i,\epsilon}_t}}\frac{\tilde\lambda^{i,\epsilon}_t}{(\tilde\lambda^{i,\epsilon}_t-\tilde\lambda^{1,\epsilon}_t\wedge\epsilon)\vee\epsilon} \nonumber
  \end{eqnarray}
  and for all $t\geq0$
  \begin{equation}
      Z(t)=\exp\left\{\int_0^t\Theta(u)\cdot d\Tilde{\textbf{B}}_u-\frac{1}{2}\int_0^t||\Theta(u)||^2du\right\}. \nonumber
  \end{equation}

  We have
  \begin{eqnarray}
    \theta_1^2(t) & = & \frac{\beta^2}{\tilde\lambda^{1,\epsilon}_t}\left(\sum_{j>1}\frac{\tilde\lambda^{1,\epsilon}_t\wedge \epsilon}{(\tilde\lambda^{j,\epsilon}_t-\tilde\lambda^{1,\epsilon}_t\wedge\epsilon)\vee\epsilon}\right)^2\nonumber\\
    & \leq & \frac{\beta^2(\tilde\lambda^{1,\epsilon}_t\wedge \epsilon)^2}{\tilde\lambda^{1,\epsilon}_t}\left(\sum_{j>1}\frac{1}{\epsilon}\right)^2 \nonumber\\
    &\leq&  \frac{(n-1)^2\beta^2}{\epsilon}.
  \end{eqnarray}

  For all $1<i\leq n$,
  \begin{eqnarray}
    \theta_i^2(t) & = & \frac{\beta^2}{\tilde\lambda^{i,\epsilon}_t}\frac{(\tilde\lambda^{i,\epsilon}_t)^2}{((\tilde\lambda^{i,\epsilon}_t-\tilde\lambda^{1,\epsilon}_t\wedge\epsilon)\vee\epsilon)^2}\nonumber\\
    & \leq & \beta^2\frac{(\tilde\lambda^{i,\epsilon}_t-\tilde\lambda^{1,\epsilon}_t\wedge\epsilon)^++\tilde\lambda^{1,\epsilon}_t\wedge\epsilon}{((\tilde\lambda^{i,\epsilon}_t-\tilde\lambda^{1,\epsilon}_t\wedge\epsilon)\vee\epsilon)^2}\nonumber\\
    &\leq &\frac{2\beta^2}{\epsilon}.
  \end{eqnarray}

  We thus have
  \begin{eqnarray}
    \mathbb{E}\left[\exp\left\{\frac{1}{2}\int_0^t||\Theta(u)||^2du\right\}\right]& <&\infty \text{ for all }t\geq0.\nonumber
  \end{eqnarray} 
  Then, according to Novikov's criterion (see for instance \cite[Proposition 5.12 p198]{MR1121940}), $Z$ is a $\mathbb{P}$-martingale, and $\mathbb{E}[Z(t)]=1$ for all $t\geq0$. Consequently, defining  for all $t\ge0$, $\tilde{\mathcal{F}}_t=\sigma\left((\tilde{\mathbf{B}}_t)_{s\le t},(\tilde\lambda_0^{1,\epsilon},\dots,\tilde\lambda_0^{n,\epsilon})\right)$ and $\mathbb{Q}$ such that
  $$\frac{d\mathbb{Q}}{d\mathbb{P}}_{|\tilde{\mathcal{F}}_t}=Z(t)$$ and
  \begin{eqnarray}
   \tilde{ \tilde {B}}^1_t&=&\tilde B^1_t-\int_0^t\theta_1(s)ds \nonumber \\
    & = & \tilde B^1_t+\int_0^t\frac{\beta}{\sqrt{\tilde\lambda^{1,\epsilon}_s}}\sum_{j\neq1}\frac{\tilde\lambda^{1,\epsilon}_s\wedge\epsilon}{(\tilde\lambda^{j,\epsilon}_s-\tilde\lambda^{1,\epsilon}_s\wedge\epsilon)\vee\epsilon}ds\nonumber\\
      \text{ for all } i\in\{2,\dots,n\},\text{ } \tilde{\tilde{ B}}^i_t  &=&\tilde B^i_t-\int_0^t\theta_i(s)ds \nonumber \\
      & = & \tilde B^i_t-\int_0^t\frac{\beta}{\sqrt{\tilde\lambda^{i,\epsilon}_s}}\frac{\tilde\lambda^{i,\epsilon}_s}{(\tilde\lambda^{i,\epsilon}_s-\tilde\lambda^{1,\epsilon}_s\wedge\epsilon)\vee\epsilon}ds, \text{ }0\leq t,\nonumber
  \end{eqnarray}
  $\Tilde{\Tilde{\textbf{B}}}=(\tilde{\tilde{ B}}^1,\dots,\tilde{\tilde{ B}}^n)$ is a $\mathbb{Q}$- Brownian motion
  according to the Girsanov theorem (see for instance \cite[Proposition 5.4 p194]{MR1121940}).

Consequently, $(B_\epsilon)$ has a global weak solution.

We now have to prove the pathwise uniqueness of the solutions to $(B_\epsilon)$.
The differences with Lemma \ref{pathwise_uniqueness} are the term \[\left[0\vee\frac{2}{\sqrt{\epsilon}}\left(\sqrt{\tilde\lambda^{i,\epsilon}_t}-\frac{\sqrt{\epsilon}}{2}\right)\wedge1\right]\] and the interaction terms between the first particle and the others. 

Let  $Z=(z^1,\dots,z^n)$ and $\tilde Z=(\tilde z^1,\dots ,\tilde z^n)$  be two global solutions to $(B_\epsilon)$ with $Z_0=\tilde Z_0$ independent from the same driving Brownian motion $\textbf{B}=(B^1,\dots,B^n)$.
   
    Let $M>0$ and \[\tau_M  =  \inf\{t\geq0 :\tilde z^1_t+ z^{1}_t+\tilde z^n_t+ z^{n}_t\geq M \}.\]
  As $Z$ and $\tilde Z$ are continuous and assumed well defined on $\mathbb{R}_+$,  $\tau_M\uparrow\infty$ when $M\uparrow\infty$.

  The local time of $z^i-\tilde z^i$ at $0$ is zero (\cite[Lemma 3.3 p389]{MR1725357}). Applying the Tanaka formula to the process $z^i-\tilde z^i$ stopped at $\tau_M$ and summing over $i$,
  \begin{eqnarray}
    \sum_{i=1}^n|z^i_{t\wedge\tau_M}-\tilde z^i_{t\wedge\tau_M}| & = & \sum_{i=1}^n|z^i_0-\tilde z^i_0|+\sum_{i=1}^n\int_0^{t\wedge\tau_M}\text{sgn}(z^i_s-\tilde z^i_s)d(z^i_s-\tilde z^i_s)\nonumber\\
    & = &2\sum_{i=1}^n\int_0^{t\wedge\tau_M}|\sqrt{z^i_s}-\sqrt{\tilde z^i_s}|dB^i_s\nonumber\\
    &&+ 2\beta\int_0^{t\wedge\tau_M}\sum_{i=2}^n\text{sgn}(z^i_s-\tilde z^i_s)\sum_{j\geq2, j\neq i}\left(\frac{z^i_s}{z^i_s-z^j_s}-\frac{\tilde z^i_s}{\tilde z^i_s-\tilde z^j_s}\right)ds\label{pu1}\\
    && -2\gamma\int_0^{t\wedge\tau_M}\sum_{i=1}^n|z^i_s-\tilde z^i_s|ds\label{pu2}\\
    &&+2\beta\int_0^{t\wedge\tau_M}\sum_{i=2}^n\Bigg\{\text{sgn}(z^i_s-\tilde z^i_s)\left(\frac{z^i_s}{(z^i_s-z^1_s\wedge\epsilon)\vee\epsilon}-\frac{\tilde z^i_s}{(\tilde z^i_s-\tilde z^1_s\wedge\epsilon)\vee\epsilon}\right)\nonumber\\
    &&-\text{sgn}(z^1_s-\tilde z^1_s)\left(\frac{z^1_s\wedge\epsilon}{(z^i_s-z^1_s\wedge\epsilon)\vee\epsilon}-\frac{\tilde z^1_s\wedge\epsilon}{(\tilde z^i_s-\tilde z^1_s\wedge\epsilon)\vee\epsilon}\right)\Bigg\}ds\label{pu3}\\
    &&-\int_0^{t\wedge\tau_M}\sum_{i=2}^n\text{sgn}(z^i_s-\tilde z^i_s)\left\{\left[0\vee\frac{2}{\sqrt{\epsilon}}\left(\sqrt{z^i_t}-\frac{\sqrt{\epsilon}}{2}\right)\wedge1\right]-\left[0\vee\frac{2}{\sqrt{\epsilon}}\left(\sqrt{\tilde z^i_t}-\frac{\sqrt{\epsilon}}{2}\right)\wedge1\right]\right\}ds\nonumber\\\label{pu4}
  \end{eqnarray}
   As the processes are stopped at $\tau_M$, the expectation of the stochastic integrals is zero.  As in the proof of Lemma \ref{pathwise_uniqueness}, the terms (\ref{pu1}) are not positive. To deal with the expectation of (\ref{pu3}), one  remarks that the function \[(x,y)\mapsto\frac{x}{(x-y\wedge\epsilon)\vee\epsilon}\] is Lipschitz of coefficient $C_M$ on $[0,M]^2$.

   As for the term (\ref{pu4}), the function $f:z\mapsto\left[0\vee\frac{2}{\sqrt{\epsilon}}\left(\sqrt{z}-\frac{\sqrt{\epsilon}}{2}\right)\wedge1\right]$ defined on $\mathbb{R}_+$ is constant on $\left[0,\frac{\epsilon}{4}\right)$ and on $[\epsilon,+\infty)$ and is differentiable on $\left(\frac{\epsilon}{4},\epsilon\right)$ with
   \begin{eqnarray}
     f'(z)  =  \frac{1}{\sqrt{\epsilon z}}\leq \frac{2}{\epsilon}&& \text{ for all }z\in\left(\frac{\epsilon}{4},\epsilon\right).\nonumber
   \end{eqnarray}
   Consequently, $f$ is Lipschitz of coefficient $\frac{2}{\epsilon}$. Then  for all $M>0$ there exists a constant $K_M\geq0$ depending on $M$ such that for all $t\geq0$
   
    \begin{eqnarray}
    \sum_{i=1}^n\mathbb{E}|z^i_{t\wedge\tau_M}-\tilde z^i_{t\wedge\tau_M}| & \leq & K_M\mathbb{E}\left[\int_0^{t\wedge\tau_M}\sum_{i=1}^n|z^i_{s}-\tilde z^i_{s}|ds\right]\nonumber\\
    & \leq & K_M\int_0^{t}\sum_{i=1}^n\mathbb{E}|z^i_{s\wedge\tau_M}-\tilde z^i_{s\wedge\tau_M}|ds.\nonumber
  \end{eqnarray}
   
 The Grönwall Lemma allows to conclude that for all $M>0$ and $t\geq0$
 \begin{equation}
    \sum_{i=1}^n\mathbb{E}|z^i_{t\wedge\tau_M}-\tilde z^i_{t\wedge\tau_M}| =0.\nonumber 
\end{equation}
  Using Fatou's Lemma to take the limit  $M$ going to infinity we deduce that for all $t\geq0$
  \begin{equation}
    \sum_{i=1}^n\mathbb{E}|z^i_{t}-\tilde z^i_{t}| =0, \nonumber 
\end{equation}
which concludes the proof on existence and pathwise uniqueness.

 Let us now prove (\ref{multiplecollisionbeta}). To do so, we can use the same method used to prove (\ref{AetaNonMultipleCollision}) on the square root of the coordinates $(\tilde \lambda^{2,\epsilon},\dots,\tilde \lambda^{n,\epsilon})$, which solve, as explained in the beginning of the proof, an SDE similar to $(A_\epsilon)$ for $n-1$ particles.

  \end{proof}

\section{Proof of the other results}

\paragraph{}

Let us now prove the result on the  of collision time between particles. To do so, we  study the difference between two neighbour coordinates and bound  it from above by a time changed Bessel process hitting zero in finite time.

\begin{proof}[Proof of Proposition \ref{everyparticlecollide}]
For $i\in\{2,\dots,n\},$
  \begin{eqnarray}
      d(\lambda_t^{i}-\lambda_t^{i-1}) & = &2\sqrt{\lambda^{i}_t}dB^{i}_t-2\sqrt{\lambda^{i-1}_t}dB^{i-1}_t+2\gamma(\lambda^{i-1}_t-\lambda^{i}_t)dt+2\beta\left(\lambda^{i}_t\sum_{j\neq i}\frac{1}{\lambda^{i}_t-\lambda^j_t}-\lambda^{i-1}_t\sum_{j\neq i-1}\frac{1}{\lambda^{i-1}_t-\lambda^j_t}\right)dt \nonumber\\
      & = & 2\sqrt{\lambda_t^{i}}dB_t^{i}-2\sqrt{\lambda_t^{i-1}}dB_t^{i-1}+2\gamma(\lambda^{i-1}_t-\lambda^{i}_t)dt+2\beta\frac{\lambda_t^{i}+\lambda_t^{i-1}}{\lambda_t^{i}-\lambda_t^{i-1}}dt+2\beta\sum_{j\neq i,i-1}\left(\frac{\lambda^{i}_t}{\lambda^{i}_t-\lambda^j_t} -\frac{\lambda^{i-1}_t}{\lambda^{i-1}_t-\lambda^j_t}\right)dt\nonumber\\
      & = &2\sqrt{\lambda_t^{i}}dB_t^{i}-2\sqrt{\lambda_t^{i-1}}dB_t^{i-1}+2\gamma(\lambda^{i-1}_t-\lambda^{i}_t)dt+2\beta\frac{\lambda_t^{i}+\lambda_t^{i-1}}{\lambda_t^{i}-\lambda_t^{i-1}}dt+2\beta\sum_{j\neq i,i-1}\lambda^{j}_t\frac{\lambda^{i-1}_t-\lambda^{i}_t}{(\lambda^{i}_t-\lambda^j_t)(\lambda^{i-1}_t-\lambda^j_t)}dt \nonumber\\
      &\leq& 2\sqrt{\lambda_t^{i}}dB_t^{i}-2\sqrt{\lambda_t^{i-1}}dB_t^{i-1}+2\beta\frac{\lambda_t^{i}+\lambda_t^{i-1}}{\lambda_t^{i}-\lambda_t^{i-1}}dt\label{diffiip1}
  \end{eqnarray}
  because 
  \begin{equation}
      2\gamma(\lambda_t^{i-1}-\lambda_t^{i})\leq 0 \nonumber
  \end{equation}
  and the contribution  of the greater and smaller coordinates is non-positive (in the sum, the numerator of the ratio factor is always non-positive and the denominator always non-negative).
   Let us fix  $i\in\{2,\dots,n\}$ and
  let us define the change of time 
  \begin{equation}
  \label{A}
      A^i_t =4\int_0^t(\lambda_s^{i}+\lambda_s^{i-1})ds \text{ for } t\geq0
  \end{equation}
  with generalized inverse
  \begin{equation}
      C^i_t=\inf\{s\geq0:A^i_s\geq t\} \nonumber.
  \end{equation}
  
 The process   $A^i$ is continuous, and according to Lemma \ref{infini} below,   $\underset{t\rightarrow\infty}\lim A^i_t=+\infty$, which implies that for all $t\geq0$, $C^i_t<\infty$.
 
 We can define
 
 \[(B^{(i)}_t)_t=\left(2\int_0^{C^i_t}\sqrt{\lambda_s^{i}}dB_s^{i}-\sqrt{\lambda_s^{i-1}}dB_s^{i-1}\right)_t.\]

  We then have for all $t\geq0$
  \begin{eqnarray}
  \langle B^{(i)}, B^{(i)}\rangle_t &=&4\langle  \int_0^{\cdot}\sqrt{\lambda_s^{i}}dB_s^{i}-\sqrt{\lambda_s^{i-1}}dB_s^{i-1},\int_0^{\cdot}\sqrt{\lambda_s^{i}}dB_s^{i}-\sqrt{\lambda_s^{i-1}}dB_s^{i-1}\rangle_{C^i_t} \nonumber\\
  & = & 4\int_0^{C^i_t}(\lambda_s^{i}+\lambda_s^{ i-1})ds \nonumber\\
  & = & A^i_{C^i_t} \nonumber\\
  & = & t.\nonumber
  \end{eqnarray}
  By Lévy's characterization, $B^{(i)}$ is a Brownian motion.
  
  We define for all $t\in\mathbb{R}_+$ : $D^i_t=\lambda_{C^i_t}^{i}-\lambda_{C^i_t}^{i-1}$.  By (\ref{diffiip1}) and the definition of $B^{(i)}$, we have :
  
  \begin{eqnarray}
      dD^i_t& =&d(\lambda_{C^i_t}^{i}-\lambda_{C^i_t}^{i-1}) \nonumber\\
      & \leq & dB^{(i)}_t+\frac{\beta}{2(\lambda_{C^i_t}^{i}-\lambda_{C^i_t}^{i-1})}dt=dB^{(i)}_t+\frac{\beta}{2D^i_t}dt.\nonumber
  \end{eqnarray}
  Let us define the Bessel process
  \begin{equation}
      r^i_t = \lambda_0^{i}-\lambda_0^{i-1}+B^{(i)}_t+\frac{\beta}{2}\int_0^t\frac{1}{r^i_s}ds. \nonumber
  \end{equation}
  Then,
  \begin{equation}
      d(r^{i}_t-D^{i}_t)\geq \frac{\beta}{2}\frac{D^{i}_t-r^{i}_t}{r^{i}_tD^{i}_t}dt\nonumber
  \end{equation}
  and  as long as neither  $r^{i}$ nor $D^{i}$ touches 0
  \begin{equation}
      d\left(e^{\frac{\beta}{2}\int_0^t\frac{ds}{r^{i}_sD^{i}_s}}(r^{i}_t-D^{i}_t)\right)\geq0\nonumber
  \end{equation}
   and  $r^{i}_t\geq D^{i}_t$ thanks to the equality $r^i_0=D^i_0$. As the trajectories are continuous, as soon as $r^{i}$ reaches $0$, which is the case in finite time almost surely when $\beta\leq1$ (see for instance \cite[Chapter XI, (ii) p. 442]{MR1725357}),  so does $D^{i}$, which concludes the proof.  
 \end{proof}
  \begin{lemma}
  \label{infini}
Let us assume $\gamma\in\mathbb{R},\alpha>0,\alpha-(n-1)\beta\geq0$, $\beta<1$ and that $\Lambda=(\lambda^1_t,\dots,\lambda^n_t)_t$ is a global solution to (\ref{eds}). Let $i\in\{2,\dots,n\}$. Then, $$ \int_0^{+\infty}(\lambda_s^{i}+\lambda_s^{i-1})ds =+\infty.$$
  \end{lemma}
  \begin{remark}
  For $i=n$ the proof is straightforward: since the coordinates are ordered, for all $t\geq0$
\[\int_0^{t}(\lambda_s^{n}+\lambda_s^{n-1})ds\geq\frac{2}{n}\int_0^{t}\left(\sum_{i=1}^n\lambda^i_s\right)ds.\]
Let us define $W$ by $W_0=0$ and
$dW_t=\sum_{i=1}^n\frac{\sqrt{\lambda^i_t}}{\sqrt{\sum_{j=1}^n\lambda^j_t}}dB^i_t$.
According to Lévy's characterization, $W$ is a Brownian motion.
Then, using the equality (\ref{sumcir}) in the introduction, $\left(\sum_{i=1}^n\lambda^i_s\right)_{s\geq0}$ is a CIR process.
  Proposition 6.2.4 in  \cite{MR2362458}  gives an expression of the Laplace transform of integrated CIR processes :  for any  $\mu>0$,
  \begin{equation*}
      \mathbb{E}\left[e^{-\mu\int_0^t\sum_{i=1}^n\lambda^i_sds}\right]  = \exp(-2\alpha\phi_\mu(t))\exp\left(-\left(\sum_{i=1}^n\lambda_0^{i}\right)\psi_\mu(t)\right)
  \end{equation*}
  where
  \begin{align*}
      \phi_\mu(t) &  = -\frac{1}{2}\ln\left(\frac{2\sqrt{\gamma^2+2\mu}e^{\left(\gamma-\sqrt{\gamma^2+2\mu}\right)t}}{\left(\sqrt{\gamma^2+2\mu}-\gamma\right)e^{-2t\sqrt{\gamma^2+2\mu}}+\sqrt{\gamma^2+2\mu}+\gamma}\right)\underset{t\rightarrow+\infty}{\longrightarrow}+\infty,\\
      \text{and }\psi_\mu(t) & = \frac{\mu\left(e^{2t\sqrt{\gamma^2+2\mu}}-1\right)}{\sqrt{\gamma^2+2\mu}-\gamma+e^{2t\sqrt{\gamma^2+2\mu}}(\sqrt{\gamma^2+2\mu}+\gamma)} \underset{t\rightarrow+\infty}{\longrightarrow} \frac{\mu}{\sqrt{\gamma^2+2\mu}+\gamma}.
  \end{align*}
  Thus,
  \[\mathbb{E}\left[e^{-\mu\int_0^{+\infty}\sum_{i=1}^n\lambda^i_sds}\right]=\underset{t\rightarrow+\infty}{\lim}\mathbb{E}\left[e^{-\mu\int_0^t\sum_{i=1}^n\lambda^i_sds}\right]=0\] so that
  $$\int_0^{+\infty}\sum_{i=1}^n\lambda^i_sds= +\infty \text{ a.s.}$$
and we can conclude.
  \end{remark}
\begin{proof}

For all $i\in\{2,\dots,n\}$, we proceed the following way.

 Let us first deal with the case $\gamma>0$.
 To do so, let $Y_0$ be distributed according to  $\rho_{inv}$ and independant from the Brownian motion $\mathbf{B}$ and let $Y= (y^1_t,\dots,y^n_t)_t$ be a solution to (\ref{eds}) starting from $Y_0$. By Proposition \ref{stationary}, for all $t\geq0$, $Y_t$ is distributed according to $\rho_{inv}$. Let us  show that \[\mathbb{P}\left(\int_0^{+\infty}y^1_sds=+\infty\right)=1.\]
Since $(Y_{t+1})_{t\geq0}$ is a  solution to (\ref{eds}) starting from $Y_1$ distributed according to $\rho_{inv}$ for the Brownian motion $(B_{t+1}-B_1)_{t\geq0}$ and pathwise uniqueness implies weak uniqueness, $(Y_{t+1})_{t\geq0}$ has the same distribution as $(Y_{t})_{t\geq0}$. Thus,
$\int_0^{+\infty}y^1_sds$ has the same distribution as $\int_1^{+\infty}y^1_sds$.
Consequently, since $$e^{-\int_0^{+\infty}y^1_sds}\leq e^{-\int_1^{+\infty}y^1_sds} \text{ a.s.},$$
\begin{eqnarray*}
     e^{-\int_0^{+\infty}y^1_sds}= e^{-\int_1^{+\infty}y^1_sds} \text{ a.s.}
\end{eqnarray*}
which also writes \[\left(1-e^{\int_0^{1}y^1_sds}\right)e^{-\int_0^{+\infty}y^1_sds}=0\text{ a.s..}\]
  As \[\rho_{inv}(\{x\in\mathbb{R}^n,x^1>0\})=1,\]we have \[y^1_0>0\text{ a.s.}\] and \[e^{\int_0^{1}y^1_sds}>1\text{ a.s.},\] and one can deduce that
  \[ e^{-\int_0^{+\infty}y^1_sds}=0\text{ a.s.}\] so that \[ \int_0^{+\infty}y^1_sds=+\infty\text{ a.s..}\]

   Let us consider a process $Z$  solution to (\ref{eds}) with integrable initial condition  with the same driving Brownian motion $\mathbf{B}$. By (\ref{contraction}), for all $t\geq0$
   \begin{equation*}
    \mathbb{E}|z^1_{t}- y^1_{t}| \leq\left(\sum_{i=1}^n\mathbb{E}|z^i_0- y^i_0|\right)\exp(-2\gamma t)
\end{equation*}
and then
\begin{align*}
    \mathbb{E}\left|\int_0^{+\infty}z^1_sds-\int_0^{+\infty}y^1_sds\right|&\leq \int_0^{+\infty}\mathbb{E}\left|z^1_s-y^1_s\right|ds\\
     & \leq \int_0^{+\infty}\sum_{i=1}^n\mathbb{E}|z^i_0- y^i_0|\exp(-2\gamma s)ds.
\end{align*}
Thus for $\gamma>0$, the right hand side is finite. Since
\[\int_0^{+\infty}y^1_sds=+\infty,\] we can deduce that \[\int_0^{+\infty}z^1_sds=+\infty,\] and with the order between the coordinates, \[\int_0^{+\infty}z^i_sds=+\infty.\]

   Let us denote the drift term of the $i$-th coordinate of (\ref{eds}) by $b_i^\gamma$. Let $\tilde\gamma\leq 0$ and $\tilde Z$  a global in time solution to (\ref{eds}) with $\gamma$ replaced by $\tilde\gamma$, with the same initial condition and the same driving Brownian motion $\mathbf{B}$ as $Z$. 
   Let $M>0$ and \[\tau_M  =  \inf\{t\geq0 :\tilde z^n_t+ z^{n}_t\geq M \}\] with the convention $\inf\emptyset= +\infty$.
  As $Z$ and $\tilde Z$ are continuous and assumed well defined on $\mathbb{R}_+$,  $\tau_M\uparrow +\infty$ when $M\uparrow\infty$. Reasoning like in the proof of Lemma \ref{pathwise_uniqueness}, we obtain
   \begin{align}
        \mathbb{E}\left[\sum_{i=1}^n(z^i_{t\wedge\tau_M}-\tilde z^i_{t\wedge\tau_M})^+\right]& = \mathbb{E}\left[\sum_{i=1}^n\int_0^{t\wedge\tau_M}\mathds{1}_{\{z^i_s>\tilde z^i_s\}}\left(b_i^\gamma(Z_s)-b_i^{\tilde{\gamma}} (\tilde Z_s) \right)ds\right]\nonumber\\
       &\leq \mathbb{E}\left[\int_0^{t\wedge\tau_M}\sum_{i=1}^n\mathds{1}_{\{z^i_s>\tilde z^i_s\}}\left(-2\gamma z^i_s+2\tilde\gamma \tilde z^i_s+\beta\sum_{j\neq i}\left(\frac{z^i_s+z^j_s}{z^i_s-z^j_s}-\frac{\tilde z^i_s+\tilde z^j_s}{\tilde z^i_s-\tilde z^j_s}\right)\right)ds\right]\nonumber\\
       &\leq \beta\mathbb{E}\left[\int_0^{t\wedge\tau_M}\sum_{i=1}^n\mathds{1}_{\{z^i_s>\tilde z^i_s\}}\sum_{j\neq i}\left(\frac{z^i_s+z^j_s}{z^i_s-z^j_s}-\frac{\tilde z^i_s+\tilde z^j_s}{\tilde z^i_s-\tilde z^j_s}\right)ds\right]\label{bigfraction}\\
       &\leq0\nonumber
   \end{align}
   The last inequality comes from the fact that, as in (\ref{sgncomputation}),
 we have for all $i<j$ :
  \begin{eqnarray}
    \left[\frac{z^i_s+z^j_s}{z^i_s-z^j_s}-\frac{\tilde z^i_s+\tilde z^j_s}{\tilde z^i_s-\tilde z^j_s}\right](\mathds{1}_{\{z^i_s>\tilde z^i_s\}}-\mathds{1}_{\{z^j_s>\tilde z^j_s\}}) & = & -2\frac{z^j_s|\tilde z^i_s-z^i_s|+z^i_s|z^j_s-\tilde z^j_s| }{(z^i_s-z^j_s)(\tilde z^i_s-\tilde z^j_s)}|\mathds{1}_{\{z^i_s>\tilde z^i_s\}}-\mathds{1}_{\{z^j_s>\tilde z^j_s\}}|\leq 0 
  \end{eqnarray}
  
  as the denominator is non-negative.

   Consequently, for all $t\geq0$, \[ \mathbb{E}\left[\sum_{i=1}^n(z^i_{t\wedge\tau_M}-\tilde z^i_{t\wedge\tau_M})^+\right]\leq0.\]
   
  Using Fatou's Lemma to take the limit  $M$ going to infinity, we deduce that for all $t\geq0$
  \begin{equation}
    \mathbb{E}\left[\sum_{i=1}^n(z^i_{t}-\tilde z^i_{t})^+\right]\leq0 \nonumber 
\end{equation}
   
   and thus,
   for all $i\in\{1,\dots,n\}$, \[z^i_t<\tilde z^i_t \text{ a.s.}.\]
   Consequently, 
   \[\int_0^{+\infty}z^i_sds=+\infty\text{ } a.s.,\] gives by comparison  \[\int_0^{+\infty}\tilde z^i_sds=+\infty\text{ } a.s.\] and the result is proved for all $\gamma\in\mathbb{R}$.

  \end{proof}

\begin{proposition}
\label{neg}
 Let us assume  $\alpha-(n-1)\beta<0$, $\beta>0$ and $0\leq\lambda^1_0\leq\dots\leq\lambda^n_0$.
 
 The SDE (\ref{eds}) has a unique strong solution defined on the time interval  $[0,\underset{\epsilon\rightarrow0}{\lim}S_\epsilon)$ where for all $\epsilon>0$,
\begin{equation}
      S_\epsilon=\inf\{t\geq0 :  \lambda^1_t\leq\epsilon\}.\nonumber
  \end{equation}
\end{proposition}
\begin{proof}
Let $\epsilon>0$ and let us consider the SDE
  \begin{equation}
  \label{negcond}
     dx^{i,\epsilon}_t  =  dB^i_t +\frac{\alpha-(n-1)\beta- 1}{2}\frac{1}{x^{i,\epsilon}_t\vee\epsilon}dt-\gamma x^{i,\epsilon}_tdt+\beta x^{i,\epsilon}_t\sum_{j\neq i}\frac{dt}{(x^{i,\epsilon}_t)^2-(x^{j,\epsilon}_t)^2}\text{ for }i\in\{1,\dots,n\}\tag{$C_\epsilon$}
    \end{equation}
    \[0\leq x^{1,\epsilon}_t<\dots< x^{n,\epsilon}_t \text{ a.s. } dt-a.e.\nonumber\]
    with initial condition $0\leq \sqrt{\lambda^1_0}\leq\dots\leq \sqrt{\lambda^n_0}$.
    
 As the function $x\mapsto\frac{1}{x\vee\epsilon}$ is globally Lipschitz, we can apply  Cepa's multivoque equations theory (\cite{MR1459451}) to conclude that there exists a unique strong solution to (\ref{negcond}) defined globally in time. Moreover, one can remark that the solution to (\ref{negcond}) is a solution to (\ref{sqrteds}) on $\left[0,\inf\{t\geq0:x^1_t\leq\epsilon\}\right]$ for all $\epsilon>0$. We can then build a solution to (\ref{eds}) up to $S_\epsilon$ for all $\epsilon>0$ by taking the square root of each coordinate. As the solutions to (\ref{eds}) are pathwise unique (see Lemma \ref{pathwise_uniqueness}), the solutions to  (\ref{negcond}) are consistent on the intervals $\left[0,\inf\{t\geq0:x^1_t\leq\epsilon\}\right]$ and we can conclude.

\end{proof}
  
  \section{Appendix}
  The next lemma deals with the existence and uniqueness to the CIR SDE and with the probability for the solution to hit zero. It is proved for instance in \cite[Theorem 6.2.2 and Proposition 6.2.3]{MR2362458}. The point $4.$ comes directly from \cite{RePEc:ecm:emetrp:v:53:y:1985:i:2:p:385-407}.
  \begin{lemma}
  \label{lamberton}
Let $a\geq0,b,\sigma\in\mathbb{R}$.  Suppose that $W$ is a standard Brownian motion defined on $\mathbb{R}_+$. For any real number $x\geq0$, there is a unique continuous, adapted process $X$, taking values in $\mathbb{R}_+$, satisfying $X_0=x$ and $$dX_t=(a-bX_t)dt+\sigma\sqrt{X_t}dW_t \text{ on } [0,\infty).$$
  Moreover, if we denote by $X^x$ the solution to this SDE starting at $x$ and by $\tau_0^x=\inf\{t\geq0 : X_t^x=0\}$,
  \begin{enumerate}
      \item If $a\geq\sigma^2/2$,  we have $\mathbb{P}(\tau^x_0=\infty)=1$, for all $x>0$.
      \item If $0\leq a <\sigma^2/2$ and $b\geq0$, we have  $\mathbb{P}(\tau^x_0<\infty)=1$, for all $x>0$.
      \item If $0\leq a <\sigma^2/2$ and $b<0$, we have  $0<\mathbb{P}(\tau^x_0<\infty)<1$, for all $x>0$.
      \item For all $s>t$,\[\mathbb{E}[r_s|r_t] = r_te^{-b(s-t)}+\frac{a}{b}(1-e^{-b(s-t)}).\] 
  \end{enumerate}
  \end{lemma}
  
      The following result is the  Ikeda-Watanabe Theorem, which allows to compare two Itô processes if their starting points and their drift coefficients are comparable, and if their diffusion coefficients  are regular enough. It is proved for instance in \cite[Theorem V.43.1 p269]{MR1780932}.
 \begin{theorem}(Ikeda-Watanabe)
 \label{Ikeda}
 Suppose that, for $i=1,2$, 
 \begin{equation}
     X^i_t = X^i_0+\int_0^t\sigma(X^i_s)dB_s+\int_0^t\beta_s^ids,
 \end{equation}
 and that there exist $b:\mathbb{R}\mapsto\mathbb{R}$, such that
 $$\beta^1_s\geq b(X^1_s)\text{, }b(X^2_s)\geq \beta^2_s.$$
 Suppose also that
 \begin{enumerate}
     \item $\sigma$ is measurable and there exists an increasing function $\rho:\mathbb{R}_+\mapsto\mathbb{R}_+$ such that $$\int_{0^+}\rho(u)^{-1}du=\infty,$$ and for all $x,y\in\mathbb{R},$
     $$(\sigma(x)-\sigma(y))^2\leq\rho(|x-y|);$$
     \item $X^1_0\geq X^2_0$ a.s.;
     \item $b$ is Lipschitz.
 \end{enumerate}
 Then $X^1_t\geq X^2_t$ for all $t$ a.s..
 \end{theorem}

\nocite{*}
\bibliographystyle{amsalpha}
\bibliography{redac}

\end{document}